\documentclass[11pt,reqno,twoside]{article}
\usepackage{amssymb,amsfonts,amsthm,amsmath}
\usepackage{color}
\usepackage[mathscr]{eucal}

\usepackage{cite}

\usepackage[left=1 in,top=1 in,right=1 in,bottom=1 in]{geometry}

\numberwithin{equation}{section}

\def\myclearpage{}

\newcommand{\beq}{\begin{equation}}
\newcommand{\eeq}{\end{equation}}
\newcommand{\beqs}{\begin{equation*}}
\newcommand{\eeqs}{\end{equation*}}
\newcommand{\ba}{\begin{array}}
\newcommand{\ea}{\end{array}}
\newcommand{\beas}{\begin{eqnarray*}}
\newcommand{\eeas}{\end{eqnarray*}}
\newcommand{\bea}{\begin{eqnarray}}
\newcommand{\eea}{\end{eqnarray}}
\newcommand{\bal}{\begin{align}}
\newcommand{\eal}{\end{align}}

\newcommand{\bals}{\begin{align*}}
\newcommand{\eals}{\end{align*}}

\newcommand{\R}{\ensuremath{\mathbb R}}

\newcommand{\norm}[1]{\| {#1} \|}

\newcommand{\bds}{\begin{displaystyle}}
\newcommand{\eds}{\end{displaystyle}}

\renewcommand{\eqref}[1]{(\ref{#1})}

\def\longequals{\mathbin{=\kern-2pt=}}
\def\eqdef{\mathbin{\buildrel \rm def \over \longequals}}

\def\varep{\varepsilon}

\def\ddt{\frac{d}{dt}}

\newcommand{\remove}[1]{} 
\renewcommand{\remove}[1]{#1} 

\newtheorem{theorem}{Theorem}[section]

\newtheorem{lemma}[theorem]{Lemma}
\newtheorem{corollary}[theorem]{Corollary}
\newtheorem{proposition}[theorem]{Proposition}

\theoremstyle{remark}

\definecolor{darkred}{rgb}{.70,.12,.20}

\definecolor{darkgreen}{rgb}{.20,.52,.14}

\newcommand{\esssup}{\mathop{\mathrm{ess\,sup}}}

\newcommand{\supp}{\operatorname{supp}}

\title{Interior Estimates for Generalized Forchheimer Flows of Slightly Compressible Fluids}

\author{Luan T. Hoang$^{a,*}$ and Thinh T. Kieu$^b$}

\date{\today}

\begin{document}
\maketitle

\begin{center}
\textit{$^a$Department of Mathematics and Statistics, Texas Tech University, Box 41042, Lubbock, TX 79409--1042, U. S. A.} \\
\textit{$^b$Department of Mathematics, University of North Georgia, Gainesville Campus, 3820 Mundy Mill Rd., Oakwood, GA 30566, U. S. A.}\\
Email addresses: \texttt{luan.hoang@ttu.edu, thinh.kieu@ung.edu}\\
$^*$Corresponding author
\end{center}

\begin{abstract}
The generalized Forchheimer flows are studied for slightly compressible fluids in porous media with time-dependent Dirichlet boundary data for the pressure. No restrictions on the degree of the Forchheimer polynomial are imposed.
We derive, for all time, the interior $L^\infty$-estimates for the pressure and its partial derivatives, and the interior $L^2$-estimates for its Hessian.
The De Giorgi and Ladyzhenskaya-Uraltseva iteration techniques are used taking into account the special structures of the equations for both pressure and its gradient. These are combined with the uniform Gronwall-type bounds in establishing the asymptotic estimates when time tends to infinity.
\end{abstract}

\pagestyle{myheadings}\markboth{L. T. Hoang and T. T. Kieu}
{Interior Estimates for Generalized Forchheimer Flows}

\myclearpage
\section{Introduction}
\label{intro}

In this paper, we study the generalized Forchheimer flows for slightly compressible fluids in porous media.
Forchheimer equations are used to describe the fluids' dynamics when the ubiquitous Darcy's law is not applicable.
They are nonlinear relations between the fluid's velocity and gradient of pressure which were realized by Darcy \cite{Darcybook} and Dupuit \cite{Dupuit1857} in early works, formulated by Forchheimer \cite{Forchh1901,ForchheimerBook}, and were studied more extensively afterward in physics and engineering, see \cite{Muskatbook,Ward64,BearBook,Nieldbook,Straughanbook} and references therein. 
Mathematics of Darcy flows has been studied thoroughly with a vast literature for a long time dated back to the 1960s, see the treaty \cite{VazquezPorousBook} for references. In contrast, Forchheimer flows were investigated mathematically much later in the 1990s.
Even less are the mathematical works on  compressible Forchheimer flows. The reader is referred to \cite{ABHI1,HI1,HI2} for more information about this topic.

Generalized Forchheimer equations were proposed and studied in our previous works \cite{ABHI1,HI1,HI2,HIKS1,HKP1} in order to cover a large class of fluid flows formulated from experiments. For compressible fluids, they form a new class of degenerate parabolic equations with their own characteristics compared to other models of porous medium equations. Among a small number of papers on these flows, recent work \cite{HKP1} is focused on studying the pressure and its time derivative in space $L^\infty$, the pressure gradient in $L^s$ for  $s\in[1,\infty)$ and the pressure's Hessian in $L^{2-\delta}$ for $\delta\in (0,1)$. However, it requires the so-called Strict Degree Condition (SDC), that is, the degree of the Forchheimer polynomial is less than $4/(n-2)$, where $n$ is the spatial dimension. Another related paper \cite{HIKS1}, in contrast, does not require (SDC), but the analysis is mainly for pressure in space $L^s$, pressure gradient in space $L^{2-a}$ for a specific number $a\in(0,1)$,  and pressure's 
time derivative in $L^2$. 
Our current work is to unite the two approaches and develop them further without any restrictions on the Forchheimer polynomials.
Specifically, we consider the initial boundary value problem (IBVP) for the pressure in a bounded domain with time-dependent Dirichlet boundary data.
The interior $W^{1,\infty}$ and $W^{2,2}$ norms of the solutions are estimated, particularly for large time. 
Such estimates for degenerate parabolic equations usually require much work, see, for e.g.,  \cite{VazquezSmoothBook,RVV2009,RVV2010,LD2000}. 
Improving on \cite{HKP1}, we adapt and refine techniques by De Giorgi  \cite{DeGiorgi57}, Ladyzhenskaya-Uraltseva \cite{LadyParaBook68}, and also DiBenedetto \cite{DiDegenerateBook} for parabolic equations, and combine them with those used for Navier-Stokes equations \cite{FMRTbook}. These techniques are utilized successfully here thanks to the special structure of our equation.

The paper is organized as follows.
We recall the model, equations and basic facts in section \ref{prelim}. 
%
In section \ref{revision}, we  derive the uniform Gronwall-type estimates that sharpen  the previous results in \cite{HI1,HIKS1} especially for large time.
In section \ref{LinftySec}, the interior $L^\infty$-estimates  for pressure  are established by using $L^\alpha$-based De Giorgi iteration with sufficiently large $\alpha$. This is different from \cite{HKP1} which is based on the $L^2$-estimate and, hence, requires (SDC).
As a result, no restrictions on the Forchheimer polynomials are needed for our estimates.
In section \ref{GradSec}, we first use Ladyzhenskaya-Uraltseva-type imbedding and iteration to estimate the pressure gradient in $L^s$ for any $s\ge 1$. 
Our estimates only require the initial data to be in the spaces $W^{1,2-a}$ and $L^\alpha$ for a fixed $\alpha$. When $s$ is large, this requirement is much less than the $W^{1,s}$ condition in \cite{HKP1}. In subsection \ref{subgrad2}, we derive the $W^{1,\infty}$ estimates for the pressure, which were not previously studied in \cite{HIKS1,HKP1}. We make use of the equation \eqref{um} for the pressure gradient, which it is highly nonlinear but has a special structure. Thanks to this and previous $W^{1,s}$ ($s<\infty$) bounds, we utilize the De Giorgi technique to obtain the $L^{\infty}$-estimates for the pressure gradient.
Section \ref{Lpt-sec} contains the $L^\infty$-estimates for  the time derivative of the pressure, while 
section \ref{Hess-sec} contains new $L^2$-bounds for the pressure's Hessian.
It is noteworthy that, thanks to the uniform Gronwall-type inequalities in section \ref{revision}, the asymptotic bounds obtained as time goes to infinity depend only on the asymptotic behavior of the boundary data.
This paper is focused primarily on the interior estimates. The issue of the solutions' estimates on the entire domain is  addressed in another work of ours \cite{HK2}.

\myclearpage
\section{Background}
\label{prelim}

Consider a fluid  in a porous medium occupying a bounded domain $U$ with boundary $\Gamma=\partial U$ in space $\R^n$. For physics problem $n=3$, but here we consider any $n\ge 2$. Let $x\in \R^n$ and $t\in\R$  be the spatial and time variables.
The fluid flow has velocity $v(x,t)\in \R^n$, pressure $p(x,t)\in \R$ and density $\rho (x,t)\in [0,\infty)$.

The generalized Forchheimer equations studied in \cite{ABHI1,HI1,HI2,HIKS1,HKP1} are of the the form:
\beq\label{gForch} g(|v|)v=-\nabla p,\eeq
where $g(s)\ge 0$ is a function defined on $[0,\infty)$.
When $g(s)=\alpha, \alpha +\beta s,\alpha +\beta s+\gamma s^2,\alpha +\gamma_m s^{m-1}$, where $\alpha,\beta,\gamma,m,\gamma_m$ are empirical constants, we have Darcy's law, Forchheimer's two-term, three-term and power laws, respectively.

In this paper, we study the case when the function $g$ in \eqref{gForch} is a generalized polynomial with non-negative coefficients, that is,
\beq\label{gsa} g(s) =a_0s^{\alpha_0}+a_1 s^{\alpha_1}+\ldots +a_N s^{\alpha_N}\quad\text{for}\quad s\ge 0,\eeq
where  
$N\ge 1$, the powers $\alpha_0=0<\alpha_1<\ldots<\alpha_N$ are fixed real numbers (not necessarily integers), the coefficients $a_0,a_1,\ldots,a_N$ are non-negative with $a_0>0$ and $a_N>0$.
This function $g(s)$ is referred to as the Forchheimer polynomial in equation \eqref{gForch}, and $\alpha_N$ is the degree of $g$.

From \eqref{gForch} one can solve $v$ implicitly in terms of $\nabla p$ and
derives a nonlinear version of Darcy's equation:
\beq\label{u-forma} v= -K(|\nabla p|)\nabla p,\eeq
where the function $K:[0,\infty)\to[0,\infty)$  is defined by
\beq\label{Kdef} K(\xi)=\frac1{g(s(\xi))},
\text{ with } s=s(\xi)\ge 0 \text{ satisfying } sg(s)=\xi, \ \text{ for }\xi\ge 0. \eeq


In addition to \eqref{gForch}, we have the continuity equation 
\beq\label{conti-eq} 
\phi\frac{\partial \rho}{\partial t}+\nabla\cdot(\rho v)=0,\eeq
where number $\phi\in(0,1)$ is the constant porosity.
Also, for slightly compressible fluids, the equation of state is 
\beq\label{slight-compress} 
\frac{d\rho}{dp}=\frac{\rho}{\kappa},\quad\text{with } \kappa=const.>0.
\eeq

From \eqref{u-forma}, \eqref{conti-eq} and \eqref{slight-compress} one derives a scalar equation for the pressure:
\beq\label{dafo-nonlin} 
\phi\frac{\partial p}{\partial t}=\kappa\nabla \cdot (K(|\nabla p|)\nabla p) + K(|\nabla p|)|\nabla p|^2.\eeq

On the right-hand side of \eqref{dafo-nonlin}, the constant $\kappa$ is very large  for most slightly compressible fluids in porous media \cite{Muskatbook}, hence we neglect its second term  and by scaling the time variable, we study the reduced equation
\beq\label{lin-p} 
\frac{\partial p}{\partial t} = \nabla\cdot (K(|\nabla p|)\nabla p).\eeq
Note that this reduction is commonly used in engineering.

Our aim is to study the IBVP for equation \eqref{lin-p} in a bounded domain.
Here afterward $U$ is a bounded open connected subset of $\mathbb{R}^n$, $n=2,3,\ldots$ with $C^2$ boundary $\Gamma=\partial U$. 
Throughout, the Forchheimer polynomial $g(s)$ is fixed.
The following number is frequently used in our  calculations:
\beq\label{ab} a=\frac{\alpha_N}{1+\alpha_N}\in(0,1).
\eeq

The function $K(\xi)$ in \eqref{Kdef}  has the following properties (c.f. \cite{ABHI1,HI1}): it is decreasing in $\xi$ mapping $\xi\in[0,\infty)$ onto $(0,1/a_0]$, and
\beq\label{K-est-3}
\frac{C_1}{(1+\xi)^a}\le K(\xi)\le \frac{C_2}{(1+\xi)^a},
\eeq
\beq\label{Kestn}
C_3(\xi^{2-a}-1)\le K(\xi)\xi^2\le C_2\xi^{2-a},
\eeq
\beq\label{K-est-2}
-a K(\xi)\le  K'(\xi)\xi\le 0,
\eeq
where $C_1$, $C_2$, $C_3$ are positive constants depending on $g$.
As in previous works, we use the function $H(\xi)$ defined by 
\beqs H(\xi)=\int_0^{\xi^2}K(\sqrt s)ds\quad \hbox{for } \xi\ge 0.\eeqs
It satisfies $K(\xi)\xi^2\leq H(\xi)\le 2K(\xi)\xi^2$, thus, by \eqref{Kestn},
\beq\label{Hcompare} C_3(\xi^{2-a}-1)\leq H(\xi)\le 2C_2\xi^{2-a}.\eeq


The following parabolic Poincar\'e-Sobolev inequalities are needed for our study. 
For each $T>0$, denote $Q_T=U\times (0,T)$.
We define a threshold exponent 
\beq\label{alphastar} 
\alpha_*=\frac{an}{2-a}.\eeq

\begin{lemma} \label{Sob4}
Assume
\beq\label{alcond}
\alpha \ge 2 \quad \text{and}\quad  \alpha>\alpha_*.
\eeq
Let 
\beq\label{expndef}
p=\alpha\Big(1+\frac{2-a}n\Big)-a.
\eeq
Then
\beq\label{nonzerobdn}
\|u\|_{L^p(Q_T)}\le C(1+\delta T)^{1/p}[[u]],
\eeq
where $\delta=1$ in general, $\delta=0$ in case $u$ vanishes on the boundary $\partial U$, and
\beq\label{udouble}
[[u]]=\esssup_{[0,T]}\|u(t)\|_{L^\alpha(U)}+\Big(\int_0^T\int_U |u|^{\alpha-2}|\nabla u|^{2-a}dx dt\Big)^\frac 1{\alpha-a}.
\eeq

In case $U=B_R$ -- a ball of radius $R$ -- the inequality \eqref{nonzerobdn} holds with 
\beq\label{uR}
[[u]]= R^{\frac{n}{p} -\frac{n}{\alpha}}\esssup_{[0,T]}\|u(t)\|_{L^\alpha(B_R)}
+R^{\frac{n}{p} - \frac{n-(2-a)}{\alpha-a}}\Big(\int_0^T\int_{B_R} |u|^{\alpha-2}|\nabla u|^{2-a}dx dt\Big)^\frac 1{\alpha-a}
\eeq
and the constant $C$ independent of $R$.
\end{lemma}

The proof of Lemma \ref{Sob4} is given in Appendix \ref{ElePfs}. 
The next is a particular embedding with spatial weights from Lemma 2.4 of \cite{HKP1} (see inequality (2.28) with $m=2$ there).

\begin{lemma}[cf. \cite{HKP1}, Lemma 2.4]\label{WSobolev}
Given $W(x,t)>0$ on $Q_T$.   
Let $r$ be a number that satisfies 
\beq\label{rcond} \frac{2n}{n+2}< r<  2.\eeq 
Set
\beq\label{sstar}
\varrho =\varrho(r)\eqdef 4(1-1/r^*).
\eeq
Then 
\beq\label{Wei1}
\|u\|_{L^\varrho (Q_T)}
\le C  [[u]]_{2,W;T} \Big\{\delta T^{\frac 1\varrho}+  \esssup_{t\in[0,T]} \Big (\int_U  W(x,t)^{-\frac r{2-r}} \chi_{{\rm supp}\,u}(x,t)dx\Big )^\frac{2-r}{\varrho r}\Big \},
\eeq
where $\delta=1$ in general, $\delta=0$ in case $u$ vanishes on the boundary $\partial U$,  and 
\beq\label{normquad}
[[u]]_{2,W;T}=\esssup_{[0,T]} \|u(t)\|_{L^2(U)}+\Big(\int_0^T \int_U W(x,t)|\nabla u|^2 dx dt\Big)^\frac 1 2.
\eeq
\end{lemma}

The following is a generalization of the convergence of fast decay geometry sequences in Lemma 5.6, Chapter II of \cite{LadyParaBook68}. It will be used in the De Giorgi iterations.

\begin{lemma}[cf. \cite{HKP1}, Lemma A.2] \label{multiseq} 
Let $\{Y_i\}_{i=0}^\infty$ be a sequence of non-negative numbers satisfying
\beqs 
Y_{i+1}\le \sum_{k=1}^m A_k B_k^i  Y_i^{1+\mu_k}, \quad 
i =0,1,2,\ldots,
\eeqs
where  $A_k>0$, $B_k>1$ and $\mu_k>0$ for $k=1,2,\ldots,m$.
Let $B=\max\{B_k : 1\le k\le m\}$ and $\mu=\min\{\mu_k : 1\le k\le m\}$. 
If $Y_0\le \min\{ (m^{-1} A_k^{-1} B^{-\frac 1 {\mu}})^{1/\mu_k} : 1\le k\le  m\}$
then $\lim_{i\to\infty} Y_i=0$.
\end{lemma}


\myclearpage
\section{Uniform Gronwall-type estimates}
\label{revision}
We study the following IBVP for $p(x,t)$:
\beq\label{p:eq}
\begin{cases}
\begin{displaystyle}
\frac{\partial p}{\partial t} 
\end{displaystyle}
 = \nabla \cdot (K (|\nabla p|)\nabla p  )& \text {in }  U\times (0,\infty),\\
p(x,0)=p_0(x) &\text {in } U,\\
p(x,t)=\psi(x,t)& \text{on } \Gamma \times(0,\infty).
\end{cases}
\eeq

In order to deal with the non-homogeneous boundary condition, the data $\psi(x,t)$ with $x\in\Gamma$ and $t>0$ is extended to a function $\Psi(x,t)$ with $x\in \bar U$ and $t\ge 0$. Throughout, our results are stated in terms of $\Psi$ instead of $\psi$. Nonetheless, corresponding results in terms of $\psi$ can be retrieved as performed in \cite{HI1}. The function $\Psi$ is always assumed to have adequate regularities for all calculations in this paper.

It is proved in Section 3 of \cite{HIKS1} that \eqref{p:eq} possesses a weak solution $p(x,t)$ for all $t>0$. 
It, in fact, has more regularity in spatial and time variables, see \cite{DiDegenerateBook}.
For the current study, we assume that solution $p(x,t)$ has sufficient regularities both in $x$ and $t$ variables such that our calculations hereafterward can be performed legitimately. 

In this section, we obtain improved estimates for solutions of \eqref{p:eq}. We emphasize the asymptotic estimates as $t\to\infty$ in terms of the asymptotic behavior of $\Psi(x,t)$.

\textbf{Notation.}
(a) Hereafter, symbol $C$ is used to denote a positive number independent of the initial and boundary data, and the time variables $t$, $T_0$, $T$;
it may depend on many parameters, namely, exponents and coefficients of polynomial  $g$,  the spatial dimension $n$ and domain $U$, other involved exponents $\alpha$, $s$, etc. in calculations. The value of $C$ may vary from place to place, even on the same line.
(b) For partial derivative notation, we will  alternatively use $\partial p/\partial t =\partial_t p=p_t$, and $\partial p/\partial x_m=\partial_m p=p_{x_m}$. 
(c) The Lebesgue norm $\|\cdot\|_{L^s}$ means $\|\cdot\|_{L^s(U)}$.
(d) For a function $f(x,t)$, we denote by $f(t)$ the function $x\to f(x,t)$. 

For $\alpha\ge 1$, we define
\beq
A(\alpha,t)=\Big[\int_U|\nabla \Psi(x,t)|^\frac{\alpha(2-a)}{2} dx\Big]^\frac {2(\alpha-a)}{\alpha(2-a)} +\Big[\int_U  |\Psi_t(x,t)|^\alpha dx\Big]^\frac{\alpha-a}{\alpha(1-a)}
\eeq
for $t\ge 0$, and
\beq A(\alpha) = \limsup_{t\to\infty} A(\alpha,t)\quad \text {and}\quad \beta(\alpha)= \limsup_{t\to\infty}[A'(\alpha,t)]^-.
\eeq
Also, define for $\alpha>0$ the number
\beq
\widehat \alpha
=\max\big\{\alpha,2,\alpha_*\big\}.
\eeq

Whenever $\beta(\alpha)$ is in use, it is understood that the function $t\to A(\alpha,t)$ belongs to $C^1((0,\infty))$.

For a function $f:[0,\infty)\to\R$, we denote by $Env f$ a continuous and increasing function $F:[0,\infty)\to\R$ such that $F(t)\ge f(t)$ for all $t\ge 0$.

Let $p(x,t)$ be a solution to IBVP \eqref{p:eq}. Denote $\overline{p} = p-\Psi$ and $\bar p_0=p_0-\Psi(\cdot,0)$.
We recall relevant results from \cite{HIKS1} below.

\begin{theorem}[cf. \cite{HIKS1}, Theorem 4.3]\label{HIKS4.3}
Let $\alpha>0$.

{\rm (i)} For all $t\ge 0$,
 \beq\label{pbar:ineq1}
 \int_U|\bar{p}(x,t)|^\alpha dx \le C\Big(1+ \int_U|\bar p_0(x)|^{\widehat \alpha} dx
+ [ Env A(\widehat\alpha,t)]^\frac{\widehat\alpha}{\widehat\alpha-a}\Big ).
 \eeq

{\rm (ii)} If $A(\widehat \alpha)<\infty$ then
\beq\label{limsupPbar}
  \limsup_{t\rightarrow\infty} \int_U |\bar{p}(x,t)|^\alpha dx \le C\big (1+A(\widehat{\alpha})^\frac {\widehat{\alpha}} {\widehat{\alpha}-a} \big).
\eeq

{\rm (iii)} If $\beta(\widehat \alpha)<\infty$ then there is $T>0$ such that
 \beq\label{pbar:ineq2}
 \int_U|\bar{p}(x,t)|^\alpha dx \le C\big (1+\beta(\widehat{\alpha})^\frac {\widehat{\alpha}}{\widehat{\alpha}-2a} + A(\widehat{\alpha},t)^\frac{\widehat{\alpha}}{\widehat{\alpha}-a} \big )\quad \text{for all }t\ge T.
 \eeq
 
\end{theorem}

For gradient and time derivative estimates, we denote
\begin{align*}
G_1(t)&=\int_U |\nabla \Psi(x,t)|^2 dx +\Big[\int_U |\Psi_t(x,t)|^{r_0} dx   \Big]^\frac{2-a}{r_0(1-a)} 
 +\Big[\int_U |\Psi_t(x,t)|^{r_0}dx  \Big]^\frac{1}{r_0},\\
G_2(t)&=\int_U|\nabla\Psi_t(x,t)|^2 dx+\int_U|\Psi_t(x,t)|^2dx,\\
G_3(t)&=G_1(t)+G_2(t),\quad
G_4(t)=G_3(t)+\int_U |\Psi_{tt}|^2 dx.
\end{align*}
with $r_0=\frac{n(2-a)}{(2-a)(n+1)-n}$.
For $t\ge 0$, recall from (4.20) in \cite{HIKS1} and from (3.25) in \cite{HI1} that
\begin{align}
\label{t0} 
\int_0^t \int_U H(|\nabla p(x,\tau)|)dx d\tau&\le C \int_U \bar p^2_0(x) dx + C\int_0^tG_1(\tau)d\tau,\\
\label{intpt}
\int_U H(|\nabla p(x,t)|)dx+\int_0^t \int_U|\bar p_t(x,\tau)|^2 dx d\tau &\le \int_U \big[H(|\nabla p_0(x)|)+\bar p^2_0(x)\big] dx +C\int_0^t G_3(\tau)d\tau. 
\end{align}

Let $\alpha\ge \widehat 2$. For $t>0$, applying Theorem 4.5 in \cite{HIKS1} with $t_0=t$ and following Remark 4.8 there to replace $\widehat 2$ by $\alpha$, we have
\begin{multline}\label{newHpt1}
\int_U |\bar p_t(x,t)|^2 dx
\le C\Big\{1+\int_U |\bar p_0(x)|^\alpha dx+ t^{-1} \Big( \int_U |\bar p_0(x)|^{2}+H(|\nabla p_0(x)|) dx + \int_0^{t}G_3(\tau)d\tau \Big)\Big\} \\
+C\int_{0}^t e^{-d_0(t-\tau)}\big( [Env A(\alpha,\tau)]^\frac{\alpha}{\alpha-a}+G_4(\tau)\big)d\tau,
\end{multline}
where $d_0>0$ is independent of $\alpha$. Since $Env A(\alpha,t)$ is increasing in $t$, and $G_3\le G_4$ we simplify \eqref{newHpt1} as
\begin{multline}\label{newHpt2}
 \int_U |\bar p_t(x,t)|^2 dx
\le C(1+t^{-1})\Big(1+\int_U |\bar p_0(x)|^\alpha dx+ \int_U H(|\nabla p_0(x)|) dx\\
+ [Env A(\alpha,t)]^\frac {\alpha}{\alpha-a}+ \int_{0}^t G_4(\tau)  d\tau\Big).
\end{multline}

Below are improved estimates for time $t$ large by using uniform Gronwall-type inequalities.

\begin{lemma}\label{UG} 
For $t\ge 1$,
\beq\label{t1}
\int_{t-1}^t \int_U H(|\nabla p(x,\tau)|)dx d\tau\le C \int_U \bar{p}^2(x,t-1) dx + C\int_{t-1}^tG_1(\tau)d\tau,
\eeq
\beq \label{t6}
\int_U H(|\nabla p(x,t)|)dx+\frac{1}{2} \int_{t-1/2}^t \int_U \bar{p}_t^2(x,\tau) dx d\tau
\le   C\int_U \bar{p}^2(x,t-1) dx + C\int_{t-1}^tG_3(\tau)d\tau,
\eeq
\beq\label{ptUni}
\int_U\bar{p}_t^2(x,t) dx\le C\int_U \bar{p}^2(x,t-1)dx+C\int_{t-1}^t G_4(\tau)d\tau. 
\eeq
\end{lemma}
\begin{proof}
The proof follows Theorems 4.4 and 4.5 in \cite{HI2}.

{\bf Proof of \eqref{t1}.} Using  inequality (3.4) of Lemma 3.1 in \cite{HI1}, one has
\beq\label{origin:eq}
\frac{d}{dt}  \int_U \bar{p}^2(x,t) dx \leq -C \int_U H(|\nabla p|)dx +CG_1(t). 
\eeq  
Dropping the negative term on the right-hand side of \eqref{origin:eq}, and integrating it from $t-1$ to $t$,  we obtain \eqref{t1}. 

{\bf Proof of \eqref{t6}.} Using (3.17) of Lemma 3.3 in \cite{HI1} with $\varep=1$, one has
\beq\label{origin:eq1}
\frac{d}{dt}  \int_U H(|\nabla p|) dx \le -\int_U \bar{p}_t^2(x,t)dx + \int_U H(|\nabla p|)dx +CG_2(t). 
\eeq
Let $s\in[t-1,t]$.  Integrating \eqref{origin:eq1} from $s$ to $t$ we have 
\begin{multline*}
\int_U H(|\nabla p(x,t)|)dx + \int_s^t \int_U \bar{p}_t^2(x,\tau) dxd\tau\\
\le \int_U H(|\nabla p|)(x,s)dx + \int_s^t \int_U H(|\nabla p|)dxd\tau +C \int_s^t G_2(\tau)d\tau.
\end{multline*}
Thus,
\begin{multline} \label{t2}
\int_UH(|\nabla p(x,t)|)dx + \int_s^t \int_U \bar{p}_t^2(x,\tau) dxd\tau\\
\le \int_U H(|\nabla p|)(x,s)dx + \int_{t-1}^t \int_U H(|\nabla p|)dxd\tau +C \int_{t-1}^t G_2(\tau)d\tau.
\end{multline}
Integrating \eqref{t2} in $s$ from $t-1$ to $t$, and using \eqref{t1}, we obtain
 \begin{multline}  \label{t5}
\int_U H(|\nabla p(x,t)|)dx+\int_{t-1}^t\int_s^t \int_U \bar{p}_t^2(x,\tau) dxd\tau ds
 \le 2\int_{t-1}^t \int_U H(|\nabla p|)dxd\tau  + C\int_{t-1}^tG_2(\tau)d\tau\\
\le C \int_U \bar{p}^2(x,t-1) dx + C\int_{t-1}^tG_1(\tau)d\tau  + C\int_{t-1}^tG_2(\tau)d\tau.
\end{multline}
Observe that    
\begin{align*}
& \int_{t-1}^t\int_s^t \int_U \bar{p}_t^2(x,\tau) dxd\tau ds  =\int_{t-1}^t\int_{t-1}^\tau \int_U \bar{p}_t^2(x,\tau) dx ds d\tau\\
& \ge \int_{t-1/2}^t\int_{t-1}^\tau \int_U \bar{p}_t^2(x,\tau) dx ds d\tau
  \ge \frac{1}{2} \int_{t-1/2}^t \int_U \bar{p}_t^2(x,\tau) dx d\tau. 
\end{align*}
Therefore \eqref{t6} follows from \eqref{t5}.   

{\bf Proof of \eqref{ptUni}.} From (3.37) of Proposition 3.11 in \cite{HI1} we have 
 \beqs
 \frac{d}{dt}\int_U\bar{p}_t^2 dx \le -C\int_U K(|\nabla p|)|\nabla p_t|^2 dx +C\int_U|\nabla \Psi_t|^2 dx -\int_U \bar{p}_t\Psi_{tt} dx.
 \eeqs
Dropping the first term on the right-hand side and using Cauchy's inequality give
 \beq\label{ptprime}
 \frac{d}{dt}\int_U\bar{p}_t^2 dx\le C\int_U|\nabla \Psi_t|^2 dx +\frac 1 2\int_U |\bar{p}_t|^2 dx+\frac 12 \int_U |\Psi_{tt}|^2 dx.
 \eeq
For $s\in[t-1/2,t]$, integrating \eqref{ptprime} from $s$ to $t$ gives    
 \begin{align*}
 \int_U\bar{p}_t^2(x,t) dx  &\le \int_U\bar{p}_t^2(x,s) dx+C\int_{s}^t\int_U|\nabla \Psi_t|^2+|\Psi_{tt}|^2 dxd\tau +\frac 1 2\int_s^t\int_U |\bar{p}_t(x,\tau)|^2dxd\tau,\\
 \int_U\bar{p}_t^2(x,t) dx &\le \int_U\bar{p}_t^2(x,s) dx+C\int_{t-1}^t\int_U|\nabla \Psi_t|^2+|\Psi_{tt}|^2 dxd\tau +\frac 1 2\int_{t-1/2}^t\int_U |\bar{p}_t(x,\tau)|^2dxd\tau.
\end{align*}
Integrating the last inequality in $s$ from $t-1/2$ to $t$ yields
\beqs
\frac12\int_U\bar{p}_t^2(x,t) dx\le \frac54\int_{t-1/2}^t\int_U\bar{p}_t^2(x,s) dx ds+C\int_{t-1}^t\int_U|\nabla \Psi_t)|^2+|\Psi_{tt}|^2 dxd\tau. 
\eeqs
Using \eqref{t6} for the first integral on the right-hand side, we have
\begin{multline}\label{thalf1}
\int_U\bar{p}_t^2(x,t) dx\le C\int_U \bar{p}^2(x,t-1)dx+C\int_{t-1}^t G_3(\tau)d\tau+C\int_{t-1}^t\int_U|\nabla \Psi_t)|^2+|\Psi_{tt}|^2 dxd\tau
. 
\end{multline} 
Using the fact $\int_U|\nabla \Psi_t|^2dx\le CG_2(t)\le CG_3(t)$, we obtain \eqref{ptUni} from \eqref{thalf1}. The proof is complete.  
\end{proof}
Combining the above, we have the following specific estimates which will be used conveniently in subsequent sections.

\begin{corollary}\label{UGcor}
Let $\alpha\ge \widehat 2$.

{\rm (i)} For $t\ge 1$,
\beq\label{all1}
\int_{t-1}^t \int_U H(|\nabla p(x,\tau)|)dx d\tau\le C\Big(1+ \int_U|\bar p_0(x)|^{\alpha} dx
+ [ Env A(\alpha,t)]^\frac{\alpha}{\alpha-a} + \int_{t-1}^tG_1(\tau)d\tau\Big ),
\eeq
\beq\label{all2}
 \int_{t-1/2}^t \int_U \bar{p}_t^2(x,\tau) dx d\tau\le C\Big(1+ \int_U|\bar{p}_0(x)|^{\alpha} dx
+ [ Env A(\alpha,t)]^\frac{\alpha}{\alpha-a} + \int_{t-1}^tG_3(\tau)d\tau\Big ).
\eeq

{\rm (ii)} If $A(\alpha)<\infty$ then
\beq\label{lim1}
  \limsup_{t\rightarrow\infty} \int_U H(|\nabla p(x,t)|)dx \le C\Big (1+A(\alpha)^\frac {\alpha} {\alpha-a} +\limsup_{t\to\infty} \int_{t-1}^tG_3(\tau)d\tau\Big),
\eeq
\beq\label{lim2}
  \limsup_{t\rightarrow\infty} \int_U\bar{p}_t^2(x,t) dx \le C\Big (1+A(\alpha)^\frac {\alpha} {\alpha-a} +\limsup_{t\to\infty} \int_{t-1}^tG_4(\tau)d\tau\Big),
\eeq
\beq\label{lim3}
  \limsup_{t\rightarrow\infty} \int_{t-1}^t \int_U H(|\nabla p(x,\tau)|)dx d\tau \le C\Big (1+A(\alpha)^\frac {\alpha} {\alpha-a} +\limsup_{t\to\infty} \int_{t-1}^tG_1(\tau)d\tau\Big),
\eeq
\beq\label{lim4}
  \limsup_{t\rightarrow\infty}  \int_{t-1/2}^t \int_U \bar{p}_t^2(x,\tau) dx d\tau \le C\Big (1+A(\alpha)^\frac {\alpha} {\alpha-a} +\limsup_{t\to\infty} \int_{t-1}^tG_3(\tau)d\tau\Big).
\eeq

{\rm (iii)} If $\beta(\alpha)<\infty$ then there is $T>0$ such that one has  for all $t\ge T$ that 
 \beq\label{large1}
 \int_U H(|\nabla p(x,t)|)dx \le C\Big (1+\beta(\alpha)^\frac {\alpha}{\alpha-2a} + A(\alpha,t-1)^\frac{\alpha}{\alpha-a} +\int_{t-1}^tG_3(\tau)d\tau\Big ),
 \eeq
 \beq\label{large2}
 \int_U\bar{p}_t^2(x,t) dx \le C\Big (1+\beta(\alpha)^\frac {\alpha}{\alpha-2a} + A(\alpha,t-1)^\frac{\alpha}{\alpha-a} +\int_{t-1}^tG_4(\tau)d\tau\Big ),
 \eeq
  \beq\label{large3}
 \int_{t-1}^t \int_U H(|\nabla p(x,\tau)|)dx d\tau \le C\Big (1+\beta(\alpha)^\frac {\alpha}{\alpha-2a} + A(\alpha,t-1)^\frac{\alpha}{\alpha-a} +\int_{t-1}^tG_1(\tau)d\tau\Big ),
 \eeq
 \beq\label{large4}
\int_{t-1/2}^t \int_U \bar{p}_t^2(x,\tau) dx d\tau \le C\Big (1+\beta(\alpha)^\frac {\alpha}{\alpha-2a} + A(\alpha,t-1)^\frac{\alpha}{\alpha-a} +\int_{t-1}^tG_3(\tau)d\tau\Big ).
 \eeq
\end{corollary}
\begin{proof} 
Note that $\widehat \alpha=\alpha$.

(i) For \eqref{all1}, resp. \eqref{all2}, we combine \eqref{t1}, resp. \eqref{t6}, with  inequality
 \beq\label{alpha2}
  \int_U \bar{p}^2(x,t-1) dx \le \int_U (1+|\bar{p}(x,t-1)|^\alpha)  dx,
 \eeq
and the estimate \eqref{pbar:ineq1} for $t-1$.

(ii) We combine estimates in Lemma \ref{UG} with \eqref{alpha2} and the limit \eqref{limsupPbar}.

(iii) This part is similar to part (ii) with the use of \eqref{pbar:ineq2} in place of \eqref{limsupPbar}.
\end{proof}

\myclearpage
\section{Interior $L^\infty$-estimates for pressure}\label{LinftySec}

In this section we estimate the $L^\infty$-norm of the pressure. \textit{Hereafter, $\alpha$ is a number that satisfies \eqref{alcond}.}
We use $\kappa_j$ to denote a number of exponents that depend on $\alpha$. Let
\beq\label{deltadef}
 \kappa_0 =\alpha(1+(2-a)/n)-a,\quad \kappa_1= 1/\delta_1 \quad\text{and}\quad \kappa_2= 1/\delta_2,
\eeq
where
$\delta_1= 1-\alpha/\kappa_0$ and $\delta_2= (1-a/\alpha)(1-\alpha_*/\alpha)$.
Note that $\delta_1,\delta_2\in (0,1)$, hence, $\kappa_1,\kappa_2>1$.

We start with estimating the $L^\infty$-norm in terms of the $L^\alpha$-norm.
\begin{theorem}\label{theo42}
Let $U'\Subset U$. If $T_0\ge 0$, $T>0$ and $\theta\in(0,1)$ then 
\beq\label{e1}
\sup_{[T_0+\theta T,T_0+T]}\norm{p(t)}_{L^\infty(U')}
\le C(1+T)^\frac{\kappa_1}{\kappa_0}\Big(1+(\theta T)^{-1}\Big)^\frac{\kappa_1}{\alpha-a} \Big(1+\| p\|_{L^\alpha(U\times(T_0, T_0+T))}\Big)^{\kappa_2}. 
\eeq
\end{theorem}
\begin{proof} Without loss of generality, we assume $T_0=0$. 
For $k\ge 0$, define $p^{(k)}=\max\{p-k,0\}$ and denote by $\chi _k(x,t)$ the characteristic function of the set $\supp p^{(k)}$.
Let $\phi_1(x)$ and $\phi_2(t)$ be cut-off functions with $\phi_1=1$ on $U'$, $\phi_1=0$ on a neighborhood of $\partial U$, and $\phi_2(0)=0$.  
Let $\zeta(x,t)=\phi_1(x)\phi_2(t)$. Multiplying the first equation in \eqref{p:eq} by $|p^{(k)}|^{\alpha-1} \zeta^2$ and integrating over $U$,  we have
\begin{align*}
& \frac1\alpha\ddt  \int_U |p^{(k)}|^\alpha\zeta^2 dx +(\alpha-1)\int_U K(|\nabla p^{(k)}|) |\nabla p^{(k)}|^2 |p^{(k)}|^{\alpha-2} \zeta^2 dx\\
&= \frac2\alpha\int_U |p^{(k)}|^\alpha \zeta \zeta_t dx +2\int_U K(|\nabla p^{(k)}|) (\nabla  p^{(k)}\cdot\nabla \zeta) |p^{(k)}|^{\alpha-1} \zeta dx.
\end{align*}
By Cauchy's inequality, we have for the last integral that
\begin{align*}
2|K(|\nabla p^{(k)}|) (\nabla  p^{(k)}\cdot\nabla \zeta) |p^{(k)}|^{\alpha-1} \zeta |
\le \frac{\alpha-1}{2} K(|\nabla p^{(k)}|)|\nabla p^{(k)}|^{2} |p^{(k)}|^{\alpha-2} \zeta^2 
+ C K(|\nabla p^{(k)}|) |p^{(k)}|^{\alpha}  |\nabla \zeta|^2.
\end{align*}
Combining the above, we obtain
\begin{align*}
&\frac1\alpha\ddt \int_U |p^{(k)}|^\alpha \zeta^2 dx+ \frac{\alpha-1}{2}\int_U K(|\nabla p^{(k)}|)|\nabla p^{(k)}|^{2} |p^{(k)}|^{\alpha-2} \zeta^2 dx \\
&  \quad\le C\int_U | p^{(k)}|^\alpha \zeta|\zeta_t| dx +C\int_U K(|\nabla p^{(k)}|)  |p^{(k)}|^{\alpha}  |\nabla \zeta|^2 dx.
\end{align*}
Using \eqref{Kestn}, we then have
\begin{align*}
&\ddt \int_U |p^{(k)}|^\alpha \zeta^2 dx+ \int_U |\nabla p^{(k)}|^{2-a} |p^{(k)}|^{\alpha-2} \zeta^2 dx \\
&  \quad\le C\int_U | p^{(k)}|^\alpha \zeta|\zeta_t| dx +C\int_U K(|\nabla p^{(k)}|) |p^{(k)}|^{\alpha}  |\nabla \zeta|^2 dx 
+ C\int_U |p^{(k)}|^{\alpha-2}  \zeta^2 dx.
\end{align*}
Using the boundedness of $K(\cdot)$ in the second to last integral, and applying Young's inequality to the last terms yield
\begin{align*}
&\ddt \int_U |p^{(k)}|^\alpha \zeta^2 dx+ \int_U |\nabla p^{(k)}|^{2-a} |p^{(k)}|^{\alpha-2} \zeta^2 dx \\
&  \quad\le C\int_U |p^{(k)}|^\alpha (\zeta^2+\zeta|\zeta_t| +|\nabla \zeta|^2) dx 
+ C\int_U \chi_k\zeta^2 dx.
\end{align*}
For $t\in[0,T]$, integrating from $0$ to $T$, and taking supremum in $t$, we obtain
\beq\label{prena}
\begin{aligned}
&  \sup_{[0,T]} \int_U |p^{(k)}|^\alpha \zeta^2 dx   + \int_0^T\int_U |\nabla p^{(k)}|^{2-a} |p^{(k)}|^{\alpha-2} \zeta^2 dx dt \\
& \le C\int_0^T \int_U |p^{(k)}|^\alpha (\zeta^2+\zeta|\zeta_t| +|\nabla \zeta|^2) dx 
+ C\int_0^T \int_U \chi_k\zeta^2 dx.\end{aligned}
\eeq

Let $x_0$ be any given point in $\bar U'$. Denote  $\rho={\rm dist}(\bar U',\partial U)>0$. 
Let $M_0 >0$ be fixed which will be determined later.
For $i\ge 0$, define 
\beq\label{indexSeq2}
k_i= M_0(1-2^{-i}),\quad
t_i =\theta T( 1- 2^{-i}),\quad 
\rho_i= \frac 1 4\rho(1+ 2^{-i}).
\eeq 
Then 
$t_0=0<t_1<\ldots<\theta T$ and $\rho_0=\rho/2>\rho_1>\ldots >\rho/4 >0.$
Note that 
\beq \label{limSeq2}
\lim_{i\to\infty}t_i=\theta T\quad\text{and}\quad 
\lim_{i\to\infty}\rho_i=\rho/4.
\eeq
  
Let $U_i =\{x : \norm{x-x_0}< \rho_i\} $ then $ U_{i+1} \Subset U_i$ for $i=0,1,2,\ldots$. 
For $i,j\ge  0$, we denote 
\beq\label{setDef2} 
\mathcal Q_i =U_i\times (t_i,T)\quad\text { and } \quad A_{i,j} =\{(x,t)\in \mathcal Q_j:  p(x,t)>k_i\}.
\eeq

For each $\mathcal Q_i$, we use a cut-off function $\zeta_i(x,t)$ with $\zeta_i\equiv 1$ in $\mathcal Q_{i+1}$ and 
$\zeta_i\equiv 0$ on $Q_T\setminus\mathcal Q_{i}$, and 
\beq\label{cutoffBound2}
|(\zeta_i)_t|\le \frac C{t_{i+1}-t_i} = \frac {C2^{i+1}}{\theta T}, \quad \text { and } \quad
|\nabla \zeta_i | \le \frac {C}{\rho_{i}-\rho_{i+1}} = \frac {C 2^{i+1}}{4\rho}.\eeq
for some $C>0$. 
Applying \eqref{prena} with $k=k_{i+1}$ and $\zeta=\zeta_i$ gives
\beq\label{pren1}
\begin{aligned}
&  \sup_{[0,T]} \int_U |p^{(k_{i+1})}|^\alpha \zeta_i^2 dx   + \int_0^T\int_U |\nabla p^{(k_{i+1})}|^{2-a}| p^{(k_{i+1})}|^{\alpha-2} \zeta_i^2 dx dt\\
& \quad\quad\le  \int_0^T\int_U |p^{(k_{i+1})}|^\alpha (\zeta_i^2+\zeta_i|\zeta_{it}|+|\nabla\zeta_{i}|^2) dxdt
 +C\mathcal  \int_0^T \int_U \chi_{k_{i+1}} \zeta_i^2 dxdt.
\end{aligned}
\eeq
Define 
\beqs
F_i\eqdef   \sup_{[t_{i+1},T]} \int_{U_{i+1}} |p^{(k_{i+1})}|^\alpha  dx   + \int_{t_{i+1}}^T\int_{U_{i+1}} |\nabla p^{(k_{i+1})}|^{2-a}|p^{(k_{i+1})}|^{\alpha-2}  dx dt.
\eeqs
Then \eqref{pren1} yields   
\beq\label{pren2}
\begin{aligned}
F_i &\le  \int_{t_i}^{T} \int_{U} |p^{(k_{i+1})}|^\alpha (\zeta_i^2+\zeta_i|\zeta_{it}|+|\nabla \zeta_i|^2) dxdt
 +C  \Big (\int_{t_{i}}^T\int_{U_i} \chi_{k_{i+1}} dx dt\Big )\\
 &\le C4^i((\theta T)^{-1}+1)  \| p^{(k_{i+1})}\|_{L^\alpha(A_{i+1,i})}^\alpha
 +C|A_{i+1,i}|.
\end{aligned}
\eeq
Since $\|p^{(k_i)}\|_{L^\alpha(A_{i,i})}\ge  \|p^{(k_i)}\|_{L^\alpha(A_{i+1,i})}\ge (k_{i+1}-k_i) |A_{i+1,i}|^{1/\alpha}$,
we have
\beq\label{a2}
  |A_{i+1,i}| \le (k_{i+1}-k_i)^{-\alpha} \| p^{(k_i)}\|_{L^\alpha(A_{i,i})}^{\alpha} \le C 2^{\alpha i} M_0^{-\alpha}\| p^{(k_i)}\|_{L^\alpha(A_{i,i})}^{\alpha}.
\eeq
This and \eqref{pren2} imply
\beq\label{Fn}
\begin{aligned}
F_i &\le  C4^i(1+(\theta T)^{-1})  \|p^{(k_{i+1})}\|_{L^\alpha(A_{i+1,i})}^\alpha
 +C 2^{\alpha i} M_0^{-\alpha}\| p^{(k_i)}\|_{L^\alpha(A_{i,i})}^{\alpha}\\
 &\le   C 2^{\alpha i}\Big(1+(\theta T)^{-1}+M_0^{-\alpha}\Big) \|p^{(k_i)}\|_{L^\alpha(A_{i,i})}^{\alpha}.
\end{aligned}
\eeq
Note that $\kappa_0$ is the exponent defined in \eqref{expndef}. Applying Lemma~\ref{Sob4} to domain $U_i$, interval $[t_{i+1},T]$ and function $p^{(k_{i+1})}$ with the use of \eqref{uR}, we have
\beq\label{a1}
 \|p^{(k_{i+1})}\|_{L^{\kappa_0}(A_{i+1,i+1})}=\|p^{(k_{i+1})}\|_{L^{\kappa_0}(\mathcal Q_{i+1})}
\le C(1+T)^\frac1{\kappa_0} (F_i^{1/\alpha}+F_i^{1/(\alpha-a)}).
\eeq
Note that the constant $C$ depends on $\rho$. 
Holder's inequality gives
\beqs
\| p^{(k_{i+1})}\|_{L^\alpha(A_{i+1,i+1})}
 \le \| p^{(k_{i+1})}\|_{L^{\kappa_0}(A_{i+1,i+1})} |A_{i+1,i+1}|^{1/\alpha-1/\kappa_0}
 \le \| p^{(k_{i+1})}\|_{L^{\kappa_0}(A_{i+1,i+1})} |A_{i+1,i}|^{1/\alpha-1/\kappa_0}.
\eeqs
It follows from this, \eqref{a1}, \eqref{Fn}, and \eqref{a2} that
\begin{align*}
 \|  p^{(k_{i+1})}\|_{L^\alpha(A_{i+1,i+1})}&
\le C (1+T)^\frac1{\kappa_0}(F_i^{1/\alpha}+F_i^{1/(\alpha-a)})  |A_{i+1,i}|^{1/\alpha-1/\kappa_0}\\
&\le C (1+T)^\frac1{\kappa_0} B^i \Big \{ \Big(1+(\theta T)^{-1}+ M_0^{-\alpha}\Big)^{1/\alpha} \|  p^{(k_i)}\|_{L^\alpha(A_{i,i})}\\
&\quad +\Big(1+(\theta T)^{-1}+M_0^{-\alpha}\Big)^{1/(\alpha-a)} \|  p^{(k_i)}\|_{L^\alpha(A_i)}^{\alpha/(\alpha-a)}
\Big\} M_0^{-1+\alpha/\kappa_0} \| p^{(k_i)}\|_{L^2(A_{i,i})}^{1-\alpha/\kappa_0},
\end{align*}
where $B=2^\alpha$. Now selecting $M_0\ge 1$ we have
\begin{align*}
\|p^{(k_i)}\|_{L^\alpha(A_{i+1,i+1})}&\le C (1+T)^\frac1{\kappa_0} B^i \Big \{ \Big(1+(\theta T)^{-1}\Big)^{1/\alpha}M_0^{-1+\alpha/\kappa_0} \|p^{(k_i)}\|_{L^\alpha(A_{i,i})}^{2-\alpha/\kappa_0}\\
 &\quad+\Big(1+(\theta T)^{-1}\Big)^{1/(\alpha-a)}M_0^{-1+\alpha/\kappa_0} \|p^{(k_i)}\|_{L^\alpha(A_{i,i})}^{\alpha/(\alpha-a)+1- \alpha/ \kappa_0}
\Big\}.
\end{align*}
Denote $ Y_i=\| p^{(k_i)}\|_{L^\alpha(A_{i,i})}.$ Then 
\beqs
Y_{i+1}\le C B^i \Big (D_1 Y_i^{1+\delta_1}+D_2Y_i^{1+\nu_1} \Big ),
\eeqs
where $\delta_1>0$ is defined in \eqref{deltadef}, $\nu_1=\alpha/(\alpha-a)-\alpha/\kappa_0>0$, 
$$D_1=(1+T)^\frac1{\kappa_0}\big(1+(\theta T)^{-1}\big)^{1/\alpha}M_0^{-\delta_1},\quad\text{and}\quad 
D_2=(1+T)^\frac1{\kappa_0}\big(1+(\theta T)^{-1}\big)^{1/(\alpha-a)}M_0^{-\delta_1}.$$
Take $M_0$ sufficiently large such that 
\beq\label{Y1}
Y_0\le C\min\{ D_1^{-1/\delta_1},D_2^{-1/\nu_1}\}.
\eeq
Then by Lemma \ref{multiseq}, $\lim_{i\to\infty} Y_i=0$, consequently,
$\int_{\theta T}^T\int_{B_{\rho/4}(x_0)} | p^{(M_0)}|^\alpha dxdt=0,$
that is
\beqs
 p(x,t)\le M_0 \quad  \text{a.e. in } B_{\rho/4}(x_0)\times (\theta T,T). 
\eeqs
Since $Y_0\le \| p\|_{L^\alpha(Q_T)}$, condition \eqref{Y1} is met if 
$\| p\|_{L^\alpha(Q_T)} \le C\min\{ D_1^{-1/\delta_1},D_2^{-1/\nu_1}\}.$
This, in turn, is satisfied if   
\beqs
M_0 \ge  C (1+T)^\frac1{\kappa_0\delta_1}\Big\{(1+(\theta T)^{-1})^{1/(\alpha\delta_1)} \| p\|_{L^\alpha(Q_T)} + (1+(\theta T)^{-1})^{1/((\alpha-a)\delta_1)} \|p\|_{L^\alpha(Q_T)}^{\nu_1/\delta_1} \Big\}.
\eeqs
Note that $\nu_1/\delta_1=1/\delta_2=\kappa_2>1$. Combining this and condition $M_0\ge 1$ we choose
\beqs
M_0=C(1+T)^\frac{\kappa_1}{\kappa_0} (1+\theta T)^{-1})^\frac{\kappa_1}{\alpha-a}(1+ \| p\|_{L^\alpha(Q_T)})^{\kappa_2},
\eeqs
with an appropriate positive constant $C$. 
By using a finite covering of $\bar U'$, we obtain 
\beqs
 p(x,t)\le M_0 \quad  \text{a.e. in } U'\times (\theta T,T). 
\eeqs
Repeating the argument for $-p$ instead of $p$, we obtain 
 $|p(x,t)|\le M_0$ a.e. in $U'\times [\theta T,T]$,  
and hence \eqref{e1} follows. The proof is complete.
\end{proof}

Combining Theorem \ref{theo42} with estimates in Section \ref{revision}, we obtain the following specific estimates for the $L^\infty$-norm.

\begin{theorem}\label{theo43} Let $U'\Subset U$.

{\rm (i)} If $t\in(0,1)$ then
 \begin{align}
\label{e2}
\norm{p(t)}_{L^\infty(U')}
&\le Ct^{-\frac{\kappa_1}{\alpha-a}}\Big(1 + \|\bar p_0\|_{L^\alpha}+ [ Env A(\alpha,t)]^\frac{1}{\alpha-a} + \|\Psi\|_{L^\alpha(U\times(0,t))}\Big )^{\kappa_2}.
\end{align}

If $t\ge 1$ then
\beq\label{pNsmall}
\norm{p(t)}_{L^\infty(U')}\le C\Big(1+\|\bar p_0\|_{L^\alpha} 
+ [ Env A(\alpha,t)]^\frac{1}{\alpha-a}+ \|\Psi\|_{L^\alpha(U\times(t-1,t))}\Big )^{\kappa_2}.
\eeq

{\rm (ii)} If $ A(\alpha)<\infty$ then
\beq\label{limsupPbarI}
  \limsup_{t\rightarrow\infty} \norm{p(t)}_{L^\infty(U')} \le C\big (1+A(\alpha)^\frac{1}{\alpha-a}+\limsup_{t\to\infty} \|\Psi\|_{L^\alpha( U\times (t-1,t) )}\big)^{\kappa_2}.
\eeq 

{\rm (iii)} If $\beta(\alpha)<\infty$ then there is $T>0$ such that
 \beq\label{PbarI}
  \norm{p(t)}_{L^\infty(U')}\le C \big (1+\beta(\alpha)^\frac 1{\alpha-2a} +\| A(\alpha,\cdot)\|_{L^\frac\alpha{\alpha-a}(t-1,t)}^\frac 1{\alpha-a} +  \|\Psi\|_{L^\alpha( U\times (t-1,t) )}\big )^{\kappa_2}
  \eeq
 for all $t\ge T$.
\end{theorem}

\begin{proof}
(i) Note that $\alpha=\widehat\alpha$. 
Let $t\in(0,1)$. Applying \eqref{e1} for $T=t$ and $\theta=1/2$, we have
\begin{align*}
\|p(t)\|_{L^\infty(U')}
&\le C(1+t^{-\frac{\kappa_1}{\alpha-a}} )(1+\|p\|_{L^\alpha(U\times(0,t))}^{\kappa_2})\\
&\le Ct^{-\frac{\kappa_1}{\alpha-a}} (1+\|\bar p\|_{L^\alpha(U\times(0,t))}+\|\Psi\|_{L^\alpha(U\times(0,t))})^{\kappa_2}.
\end{align*}
Using \eqref{pbar:ineq1} to estimate $\|\bar p\|_{L^\alpha(U\times(0,t))}$ and the fact that $Env A(\alpha,t)$ is increasing in $t$, we obtain \eqref{e2}.   

For $t\ge 1$, applying \eqref{e1}  with $T_0=t-1$, $T=1$ and $\theta=1/2$ we obtain 
\beq\label{pinf}
\norm{p(t)}_{L^\infty(U')}\le C\Big(  1 +  \|p\|_{L^\alpha( U\times (t-1,t) )}^{\kappa_2}\Big)
\le C\Big(  1 +  \|\bar p\|_{L^\alpha( U\times (t-1,t) )}+  \|\Psi\|_{L^\alpha( U\times (t-1,t) )}\Big)^{\kappa_2}. 
\eeq
Again, using inequality \eqref{pbar:ineq1}  to estimate $\|\bar p\|_{L^\alpha(U\times(t-1,t))}$, we obtain \eqref{pNsmall}.   
 
(ii) From \eqref{pinf} we have
\beqs
\limsup_{t\to\infty} \norm{p(t)}_{L^\infty(U')}\le C\Big( 1 +  \limsup_{t\to\infty} \|\bar p\|_{L^\alpha( U\times (t-1,t) )}
+  \limsup_{t\to\infty} \|\Psi\|_{L^\alpha( U\times (t-1,t) )}\Big)^{\kappa_2}. 
\eeqs
By \eqref{limsupPbar},    
\beqs
\limsup_{t\to\infty}  \| \bar p\|_{L^\alpha( U\times (t-1,t) )}^\alpha \le \limsup_{t\to\infty} \int_U |\bar p(x,t)|^\alpha dx\le  C(1+A(\alpha))^{\alpha/(\alpha -a)}.
\eeqs
Thus \eqref{limsupPbarI} follows. 

(iii) By \eqref{pbar:ineq2}, we have for large $t$ that
 \beq\label{fix1}
 \int_{t-1}^t\int_U|\bar{p}(x,\tau)|^\alpha dx d\tau\le C\big (1+\beta(\alpha)^\frac {\alpha}{\alpha-2a} 
 + \int_{t-1}^t A(\alpha,\tau)^\frac{\alpha}{\alpha-a}d\tau).
 \eeq
From \eqref{pinf} and \eqref{fix1}, we obtain \eqref{PbarI}.  
\end{proof}

It is worthy to mention that the quantities on the right-hand sides of estimates \eqref{limsupPbarI} and \eqref{PbarI} only depend on the boundary data's large time behavior.
This feature will be re-established in many of our estimates below for the gradient, time derivative and Hessian of the pressure.

%
%
%

\myclearpage
\section{Interior estimates for pressure gradient}\label{GradSec}

We will first estimate the pressure gradient in $L^s$-norm (for $s<\infty$) in subsection \ref{subgrad1}, and then in $L^\infty$-norm in subsection \ref{subgrad2}.

\subsection{ $L^s$-estimates}\label{subgrad1}
The following imbedding lemma from \cite{HKP1} is a suitable extension of Lemma 5.4 on page 93 in \cite{LadyParaBook68}.

\begin{lemma}[cf. \cite{HKP1}, Lemma 3.2]\label{LUK} For each $s\geq1$, there exists a constant $C>0$ depending on $s$ such that  for each smooth cut-off function 
$\zeta(x) \in C_c^\infty(U)$, the following inequality holds 
\begin{equation*}
\begin{aligned}
\int_U K(|\nabla p|) |\nabla p|^{2s+2} \zeta^2 dx
& \le  C \max_{\text{\rm supp} \zeta } |p|^2  \Big [ \int_U K(|\nabla p|)|\nabla p|^{2s-2} |\nabla^2 p|^2 \zeta^2 dx \\
 &\quad \quad \quad +  \int_U K(|\nabla p|)|\nabla p|^{2s}  |\nabla \zeta|^2 dx \Big],
\end{aligned} 
\end{equation*}
for every sufficiently regular function $p(x)$ such that the right hand side is well-defined.
\end{lemma}

We establish the basic step for the Ladyzhenskaya-Uraltseva iteration.

\begin{lemma} \label{new.it-L} For each $s\geq 0$, if $T_0\ge 0$, $T> 0$, and $\zeta(x,t)$ is a smooth function  with $\zeta(\cdot,T_0)=0$ and $\supp \zeta(\cdot,t)\subset U$ for all $t\in[T_0,T_0+T]$, then
\begin{multline}\label{eLs2}
  \sup_{[T_0,T_0+T]} \int_U |\nabla p(x, t)|^{2s+2} \zeta^2 dx
+\int_{T_0}^{T_0+T}\int_U K(|\nabla p|) |\nabla^2 p|^2   |\nabla p|^{2s}  \zeta^2 dx dt \\
\le   C \int_{T_0}^{T_0+T} \int_U K(|\nabla p|)  |\nabla p|^{2s+2}|\nabla \zeta|^2 dx dt+C\int_{T_0}^{T_0+T} \int_U |\nabla p|^{2s+2}\zeta |\zeta_t| dx.
\end{multline}
 \end{lemma}
\begin{proof} Without loss of generality, assume $T_0=0$. 
Multiplying the equation \eqref{p:eq} by 
$-\nabla \cdot (|\nabla p|^{2s} \zeta^2 \nabla p )$, integrating the resultant over $U$ and using integration by parts, we obtain
\begin{align*}
 \frac{1}{2s+2} \ddt \int_U |\nabla p|^{2s+2} \zeta^2 dx
= -\int_U \partial_j (K(|\nabla p|)\partial_i p) \, \partial_i (|\nabla p|^{2s} \partial_j p \zeta^2) dx +\frac 1{s+1} \int_U |\nabla p|^{2s+2}\zeta \zeta_t dx .
\end{align*}
Calculated as in Lemma 3.1 of \cite {HKP1} above equation is rewritten as     
\begin{align*}
& \frac{1}{2s+2} \ddt \int_U |\nabla p|^{2s+2} \zeta^2 dx
= -\int_U \Big[\partial_{y_l} (K(|y|) y_i)\Big|_{y=\nabla p} \partial_j\partial_l p\Big]\partial_j \partial_i p  |\nabla p|^{2s}  \zeta^2 dx\\
&\quad -2\int_U\Big[ \partial_{y_l} (K(|y|) y_i)\Big|_{y=\nabla p} \partial_j\partial_l p \Big] \partial_j p \, |\nabla p|^{2s} \, \zeta \partial_i  \zeta dx\\
&\quad -2s\int_U \Big[\partial_{y_l} (K(|y|) y_i)\Big|_{y=\nabla p} \partial_j\partial_l p \Big]\partial_j p   \, (|\nabla p|^{2s-2} \partial_i\partial_m p\partial_m p)\,     \zeta^2 dx+\frac 1{s+1} \int_U |\nabla p|^{2s+2}\zeta \zeta_t dx.
\end{align*}	
We denote the four terms  on the right-hand side by $I_1$, $I_2, I_3$ and $I_4$. 
It follows from the calculations in Lemma 3.1 of \cite {HKP1} that  
\begin{align*}
 I_1&\le -(1-a)\sum_j \int_U K(|\nabla p|) |\nabla(\partial_j p)|^2   |\nabla p|^{2s}  \zeta^2 dx,\\
 |I_2|&\le 2(1+a)\int_U K(|\nabla p|) |\nabla^2 p| |\nabla p|^{2s+1} \zeta  |\nabla \zeta| dx, 
\end{align*}
and
\begin{align*}
 I_3\le  -2(1-a)s\int_U K(| \nabla p|)\Big | \nabla \Big(\frac12|\nabla p|^2\Big)\Big |^2  |\nabla p|^{2s-2}   \zeta^2 dx\le 0.
\end{align*}
Combining these estimates with Young's inequality, we find that 
\beq\label{gradid}
\begin{aligned}
 & \frac{1}{2s+2} \ddt \int_U |\nabla p|^{2s+2} \zeta^2 dx
+\frac{1-a}{2}  \int_U K(|\nabla p|) |\nabla^2 p|^2   |\nabla p|^{2s}  \zeta^2 dx\\
&\quad\le  C \int_U K(|\nabla p|) |\nabla p|^{2s+2}|\nabla \zeta|^2 dx+\frac 1{s+1} \int_U |\nabla p|^{2s+2}\zeta |\zeta_t| dx.
\end{aligned}
\eeq
Inequality \eqref{eLs2} follows directly by integrating \eqref{gradid} from $0$ to $T$.
\end{proof}

Next, we reduce estimates for $W^{1,s}$-norm, with large $s$, down to $W^{1,2-a}$ and $L^\infty$ norms. 

\begin{proposition}\label{Prop53}
Let $U'\Subset V\Subset U$, $T_0\ge 0$, $T>0$ and $\theta\in (0,1)$.
If $s\ge  2$ then
 \begin{multline}\label{new0}
\int_{T_0+\theta T}^{T_0+T} \int_{U'} K(|\nabla p|) |\nabla p|^{s} dx dt 
\le C\Big(1+(\theta T)^{-1}\Big)^{s-2} \big(1+\sup_{[T_0+\theta T/2,T_0+T]} \| p\|^2_{L^\infty(V)} \big)^{s-2} \\
\cdot\int_{T_0}^{T_0+T}  \int_{U} (1+K(|\nabla p|) |\nabla p|^2) dxdt,
\end{multline}
and
\begin{multline}\label{new1}
\sup_{t\in [T_0+\theta T,T_0+T]}\int_{U'}|\nabla p(x,t)|^s dx dt
 \le C\Big(1+(\theta T)^{-1}\Big)^{s+a-1}\big(1+\sup_{[T_0+\theta T/2,T_0+T]} \|p\|^2_{L^\infty(V)} \big)^{s-2+a}\\
  \cdot\int_{T_0}^{T_0+T} \int_{U} (1+K(|\nabla p|) |\nabla p|^2) dxdt,
\end{multline}
where constant $C>0$ is independent of $T_0$, $T$, and $\theta$.
\end{proposition}
\begin{proof} 
Without loss of generality, assume $T_0=0$.
First, we prove a more general version of \eqref{new0}.

\textit{Claim.} For $0<\theta'<\theta<1$ we have
 \begin{multline}\label{new2}
\int_{\theta T}^{T} \int_{U'} K(|\nabla p|) |\nabla p|^{s} dx dt 
\le C\Big(1+[(\theta-\theta') T]^{-1}\Big)^{s-2} \big(1+\sup_{[\theta' T,T]} \| p\|^2_{L^\infty(V)} \big)^{s-2} \\
\cdot\int_{0}^{T}  \int_{U} (1+K(|\nabla p|) |\nabla p|^2) dxdt,
\end{multline}
where constant $C>0$ is independent of $T$, $\theta$, and $\theta'$.

Note that \eqref{new2} holds trivially when $s=2$. Hence, we will only prove \eqref{new2} for $s>2$ in a number of steps. 

\textbf{Step 1.} Let $\zeta(x,t)$ be the cut-off function with $\zeta(x,t)=0$ for $(x,t)\not\in V\times (\theta' T,T]$ and $\zeta(x,t)=1$ for $(x,t)\in U'\times (\theta T,T]$. By applying Lemma \ref{LUK}  with $s+1$ in place of $s$, we have for $s\ge 0$ that
\begin{align*} 
\int_0^T \int_U K(|\nabla p|) |\nabla p|^{2s+4} \zeta^2 dx dt 
\le  C \max_{ \rm{supp}\zeta } |p|^2 & \Big[ \int_0^T\int_U K(|\nabla p|)|\nabla p|^{2s} |\nabla^2 p|^2 \zeta^2 dx dt \\
& +  \int_0^T \int_U K(|\nabla p|)|\nabla p|^{2s+2}  |\nabla \zeta|^2 dx dt \Big].
\end{align*}
Let $N_0=\sup_{V\times[\theta' T,T]}|p|^2$. Using \eqref{eLs2} to estimate the first integral on the right-hand side, we find that      
\beq\label{ladyural4}
\begin{aligned} 
&\int_0^T \int_U K(|\nabla p|) |\nabla p|^{2s+4} \zeta^2 dx dt \\
&\quad \le  CN_0\Big [ \int_0^T \int_U K(|\nabla p|)  |\nabla p|^{2s+2}|\nabla \zeta|^2 dx dt+\int_0^T \int_U |\nabla p|^{2s+2}\zeta |\zeta_t| dx \Big].
\end{aligned}
 \eeq 
By the boundedness and property \eqref{K-est-3} of function $K(\xi)$, 
 \begin{multline}\label{Ks}
K(|\nabla p|)  |\nabla p|^{2s+2}\le C  |\nabla p|^{2s+2} \le C K(|\nabla p|)(1+|\nabla p|^a)  |\nabla p|^{2s+2} \\
 \le  CK(|\nabla p|)(1+  |\nabla p|^{2s+2+a})\le C(1+K(|\nabla p|)  |\nabla p|^{2s+2+a}).
 \end{multline}
Hence, inequality \eqref{ladyural4} leads to  
 \beq\label{Kiter}
\int_0^T \int_{U} K(|\nabla p|) |\nabla p|^{2s+4} \zeta^2 dx dt
 \le  C N_0\int_0^T \int_{U} (1+ K(|\nabla p|)  |\nabla p|^{2s+2+a}) (|\nabla \zeta|^2 +\zeta|\zeta_t|) dx dt.
\eeq
%


{\bf Step 2.} Let $s=0$ in \eqref{ladyural4}. For $\varep>0$, applying Cauchy's inequality to the last integral of \eqref{ladyural4} gives
\begin{multline}\label{Kgrad3}
 \int_0^T \int_{U} K(|\nabla p|) |\nabla p|^4 \zeta^2 dx
 \le \varep N_0\int_0^T \int_{U}  K(|\nabla p|)|\nabla p|^4 \zeta^2 dx
           + C\varep^{-1} N_0  \int_0^T \int_{U} K(|\nabla p|)^{-1} \zeta_t^2 dx dt \\
+C N_0 \int_0^T \int_{U} K(|\nabla p|) |\nabla p|^2 |\nabla \zeta|^2  dx dt\Big].
\end{multline}
Here, the cut-off function $\zeta$  satisfies $|\nabla \zeta|\le C$ and $0\le \zeta_t\le C[(\theta-\theta')T]^{-1}$.
Taking $\varep=(2(N_0+1))^{-1}$ in \eqref{Kgrad3} and using Young's inequality yield 
\begin{align*}
  \int_{\theta T}^T\int_{U'} K(|\nabla p|)|\nabla p|^4 dxdt  
& \le C(N_0+1)N_0((\theta-\theta') T)^{-2} \int_{\theta' T}^T\int_{U}  (1+|\nabla p|)^{a}  dx dt \\
&\quad + C N_0 \int_{\theta' T}^T \int_{U} K(|\nabla p|) |\nabla p|^2 dx dt\\
& \le C(1+N_0)^2\Big(1+\frac 1{(\theta-\theta') T}\Big)^2 \int_{\theta' T}^T\int_{U}  (1+K(|\nabla p|) |\nabla p|^2)   dx dt  .
\end{align*}
This implies \eqref{new2} when $s =4$. 

{\bf Step 3.}  When $s \in (2,4)$, let $\beta$ be a number in $(0,1)$ such that 
$\frac 1 {s} =\frac{1-\beta}{2}+\frac\beta {4}$. 
Then, by interpolation inequality:  
\begin{align*}
\Big(\int_{\theta T}^T\int_{U'} K(|\nabla p|)|\nabla p|^{s} dx dt\Big)^{\frac 1 {s}} &\le \Big(\int_{\theta T}^T\int_{U} K(|\nabla p|) |\nabla p|^{2} dx dt\Big)^\frac{1-\beta}{2} \Big(\int_{\theta T}^T\int_{U'} K(|\nabla p|) |\nabla p|^{4}dx dt\Big)^{\frac \beta {4}}.
\end{align*}
Note that $\beta s/4=s/2-1$. Using \eqref{Kgrad3} to estimate the last double integral, we obtain
\beq\label{Kgrad24}
\int_{\theta T}^T\int_{U'} K(|\nabla p|)|\nabla p|^{s} dx dt
\le C \Big(1+\frac{1}{(\theta-\theta') T}\Big)^{s-2}(1+N_0)^{s-2}  \int_{\theta' T}^T \int_{U}(1+ K(|\nabla p|) |\nabla p|^2) dx.
\eeq
This implies \eqref{new2} for $s\in(2,4)$. Therefore, we have proved \eqref{new2} for $s \in (2, 4]$.

{\bf Step 4.} When $s>4$, let $m$ be the positive integer that satisfies $\frac {s-4}{2-a}\le m<\frac {s-4}{2-a}+1$, then $s- m(2-a)\in (2+a,4] $.  
Then, let $\{U_k\}_{k=0}^m$ and $\{V_k\}_{k=0}^m$ be two families of open subsets of $U$ such that  
 \beqs
 U'=U_0\Subset V_0\Subset U_1\Subset V_1\Subset U_2\Subset  V_2\Subset \ldots \Subset U_{m-1} \Subset V_{m-1}\Subset U_m \Subset V_m= V\Subset U.
 \eeqs
For each $k =0,1,\ldots, m+1$, let $\theta_k=\theta-(\theta-\theta')k/(m+1)$. 
Note that $\theta=\theta_0>\theta_1>\theta_2>\ldots>\theta_m>\theta_{m+1}=\theta'$ and $\theta_k-\theta_{k+1}=(\theta-\theta')/(m+1)$.
For $1\le k\le m$, let $\zeta_k(x,t)$ be a smooth cut-off function 
 which is equal to one on $U_k\times(t_k,T] $ and zero on $Q_T\setminus V_k\times(t_{k-1},T]$, and satisfies
 $ |\nabla \zeta_k| \le C_m$ and $0\le \zeta_{k,t}\le  C_m[(\theta-\theta') T]^{-1}$.
 Here and in the following calculations in this proof,  $C_m$ denotes a generic positive constant depending on all $U_k$ and $V_k$, $k=1,2,\ldots,m$. 

We denote $N_0=\sup_{V\times[\theta' T,T]}|p|^2$ still.
Applying  \eqref{Kiter} to $V_0$, $U_1$, $\theta_1$ in place of $V$, $U$, $\theta$, we have 
\begin{align*}
\int_{\theta T}^T\int_{U'} K(|\nabla p|) |\nabla p|^{s} dx dt 
&\le C_m (1+[(\theta-\theta') T]^{-1}) N_0 \int_{\theta_1 T}^T \int_{U_1} (1+K(|\nabla p|) |\nabla p|^{s-(2-a)}) dx dt \\
&= C_m (1+[(\theta-\theta') T]^{-1})N_0\Big[ T+  \int_{\theta_1 T}^T\int_{U_1}K(|\nabla p|) |\nabla p|^{s-(2-a)} dx dt\Big].
\end{align*}
Above, we used the fact that $\sup_{V_k\times [\theta_{k-1} T,T]}|p|^2\le N_0$ for $1\le k\le m$.
Applying  \eqref{Kiter} recursively to  $U_k\Subset V_k \Subset U_{k+1}$ and $\theta_k>\theta_{k+1}$, we obtain   
\begin{align*}
\int_{\theta T}^T\int_{U'} K(|\nabla p|) |\nabla p|^{s} dx dt 
& \le  C_mT\sum_{i=1}^m  d^{i} +C d^m\int_{\theta_m T}^T \int_{U_m} K(|\nabla p|) |\nabla p|^{s-m(2-a)} dx dt\\
& \le C_m d^m \Big[ T+\int_{\theta_m T}^T \int_{U_m} K(|\nabla p|) |\nabla p|^{s-m(2-a)} dx dt\Big],
\end{align*}
where $d=(1+[(\theta-\theta') T]^{-1})(1+N_0)$.
Since $s-m(2-a)\in(2,4]$, estimating the last integral by \eqref{Kgrad24} gives
\beq\label{grad100}
\int_{\theta T}^T \int_{U'} K(|\nabla p|) |\nabla p|^{s} dx dt 
\le Cd^m T +Cd^{s-m(2-a)-2+m} \int_{\theta' T}^T \int_{U} K(|\nabla p|) |\nabla p|^2 dx.
\eeq
Since $d\ge 1$ and $m,s-m(2-a)-2+m \le  s-2$, the desired estimate \eqref{new2} follows \eqref{grad100}. 
This completes the proof of \eqref{new2} for all $s\ge 2$.

The inequality \eqref{new0} then is an obvious consequence of \eqref{new2} with $\theta'=\theta/2$.
  
Now, we turn to the proof of \eqref{new1}. For $s>2-a$, it follows from \eqref{eLs2} that 
\begin{align*}
\sup_{[0,T]} \int_U |\nabla p|^s \zeta^2 dx
&\le   C \int_0^T \int_U K(|\nabla p|)  |\nabla p|^s|\nabla \zeta|^2 dx dt+C\int_0^T \int_U |\nabla p|^s\zeta |\zeta_t| dx\\
&\le   C \int_0^T \int_U |\nabla p|^s( |\nabla \zeta|^2 + \zeta |\zeta_t|) dxdt.
\end{align*}
Same as \eqref{Ks},  we have $|\nabla p|^s \le  C (1+K(|\nabla p|) |\nabla p|^{s+a})$, and hence
\beq\label{new3}
\sup_{[0,T]} \int_U |\nabla p|^s \zeta^2 dx
\le  C\int_0^T \int_U (1+K(|\nabla p|)|\nabla p|^{s+a})(|\nabla \zeta|^2+\zeta |\zeta_t|) dx.
\eeq 
Let $V'$ be an open  set such that $U'\Subset V'\Subset V$. Choose $\zeta(x,t)$ such that $\zeta=0$ for $t\le 3\theta T/4$ or $x\not\in V'$, and $\zeta =1$ for $(x,t)\in U'\times [\theta T,T]$.
Then we have from \eqref{new3} that
\beq\label{new4}
\sup_{[\theta T,T]} \int_U |\nabla p|^s  dx
\le   C (1+(\theta T)^{-1})\int_{3\theta T/4}^T \int_{V'} (1+K(|\nabla p|)|\nabla p|^{s+a}) dx dt.
\eeq 
To estimate the last integral, we apply \eqref{new2} with parameters $s$, $\theta'$, $\theta$, $U'$ being replaced by $s+a$, $\theta/2$, $3\theta/4$, $V'$. 
Therefore, we obtain
\begin{align*}
\sup_{[\theta T,T]} \int_{U'} |\nabla p|^s  dx
&\le   C\Big\{((1+(\theta T)^{-1})^{s+a-1} (1+N_0)^{s+a-2} \int_0^T \int_{U} 1+K(|\nabla p|)|\nabla p|^2dx dt\Big\},
\end{align*} 
hence proving \eqref{new1}. The proof is complete.
\end{proof}

Now, we combine Proposition \ref{Prop53} with estimates in section \ref{revision} to express the bounds in terms of the initial and boundary data.

\begin{theorem}\label{gradsthm}
Let $U'\Subset U$ and $s\ge 2$. If $t\in(0,2)$ then
\begin{multline}\label{gradps-c}
\int_{U'} |\nabla p(x,t)|^s dx 
\le C t^{-\mu_1}(1+\|\bar p_0\|_{L^\alpha})^{\mu_2+2} \Big(1+ [ Env A(\alpha,t)]^\frac{1}{\alpha-a}+ \|\Psi\|_{L^\alpha(U\times(0,t))}\Big )^{\mu_2}\\
\quad \cdot\Big( 1+\int_0^t G_1(\tau)d\tau\Big),
\end{multline}
where 
\beq\label{mu12}
\mu_1=\mu_1(\alpha,s)\eqdef \big[1+\frac{2\kappa_1}{\alpha-a}\big](s+a-2)+1
\quad\text{and}\quad \mu_2=\mu_2(\alpha,s)\eqdef 2\kappa_2 (s-2+a).
\eeq
If $t\ge 2$ then
\begin{multline}\label{gradps-b}
\int_{U'} |\nabla p(x,t)|^s dx 
\le C (1+\|\bar p_0\|_{L^\alpha})^{\mu_2+\alpha}  \Big(1+ [ Env A(\alpha,t)]^\frac{1}{\alpha-a}+ \|\Psi\|_{L^\alpha(U\times(t-2,t))}\Big )^{\mu_2+\alpha}\\
 \cdot\Big( 1+\int_{t-1}^t G_1(\tau)d\tau\Big).
\end{multline}
\end{theorem}
\begin{proof}
 Let $t\in(0,2)$. Applying \eqref{new1} to $T_0=0$, $T=t$ and $\theta=1/2$, and using \eqref{e2}, \eqref{t0} and relation \eqref{Hcompare}, we obtain 
 \begin{multline*}
\int_{U'} |\nabla p(x,t)|^s dx 
\le C t^{-\mu_1}\cdot \Big(1+\|\bar p_0\|_{L^\alpha} 
+ [ Env A(\alpha,t)]^\frac{1}{\alpha-a}+ \|\Psi\|_{L^\alpha(U\times(0,t))}\Big )^{2\kappa_2(s-2+a)}\\
\quad \cdot\Big( 1+\|\bar p_0\|_{L^2}^2 +\int_0^t G_1(\tau)d\tau\Big).
\end{multline*}
Then \eqref{gradps-c} follows.
 Let $t\ge 2$. Applying \eqref{new1} with $T_0=t-1$, $T=1$, $\theta=1/2$, then combining it with \eqref{pNsmall} and \eqref{all1}, we obtain
 \begin{multline*}
\int_{U'} |\nabla p(x,t)|^s dx 
\le C \Big(1+\|\bar p_0\|_{L^\alpha} 
+ [ Env A(\alpha,t)]^\frac{1}{\alpha-a}+ \|\Psi\|_{L^\alpha(U\times(t-2,t))}\Big )^{\mu_2}\\
 \cdot\Big( 1+\|\bar p_0\|_{L^\alpha}^\alpha +Env A(\alpha,t)^\frac{\alpha}{\alpha-a}+\int_{t-1}^t G_1(\tau)d\tau\Big).
\end{multline*}
Then  \eqref{gradps-b} follows.
\end{proof}

For large time estimates, we have:

\begin{theorem}\label{theo55}
Let $U'\Subset U$, $s\ge 2$, and let $\mu_2$ be defined as in Theorem \ref{gradsthm}.
  
{\rm (i)} If $A(\alpha)<\infty$ then
\beq\label{LsupGradp-s}
\limsup_{t\to\infty}\int_{U'} |\nabla p(x,t)|^s dx \le C \Big (1+A(\alpha)^\frac{1}{\alpha-a}+\limsup_{t\to\infty} \|\Psi\|_{L^\alpha( U\times (t-1,t))}\Big)^{\mu_2+\alpha}
\Big(1+ \limsup_{t\to\infty} \int_{t-1}^tG_1(\tau)d\tau\Big).
\eeq

{\rm (ii)} If $\beta(\alpha)<\infty$ then there is $T>0$ such that for all $t>T$,
\beq\label{Gradp-sL}
\begin{aligned}
\int_{U'} |\nabla p(x,t)|^s dx &\le C \Big( 1 + \beta(\alpha)^\frac{1}{\alpha-2a} + \sup_{[t-2,t]}A(\alpha,\cdot)^\frac{1}{\alpha-a} 
 + \|\Psi\|_{L^\alpha( U\times (t-2,t) )}\Big)^{\mu_2+\alpha}
\Big(1+\int_{t-1}^t G_1(\tau)d\tau\Big).
\end{aligned}
\eeq
\end{theorem}
\begin{proof} For $t\ge 1$
Applying \eqref{new1} with $T_0=t-1$, $T=1$, $\theta=1/2$, 
  \beq\label{Grads0}
  \begin{aligned}
  &\int_{U'}|\nabla p(x,t)|^s dx dt
  \le C\big(1+\sup_{[t-1,t]} \|p\|^2_{L^\infty(V)} \big)^{s-2+a} \cdot \int_{t-1}^t \int_{U} (1+K(|\nabla p|) |\nabla p|^2) dxdt.
  \end{aligned}
  \eeq

(i) Taking limit superior as $t\to\infty$ and using \eqref{limsupPbarI}, \eqref{lim3} give
\begin{align*}
\limsup_{t\to\infty}\int_{U'} |\nabla p(x,t)|^s dx 
&\le C \big (1+A(\alpha)^\frac{1}{\alpha-a}+\limsup_{t\to\infty} \|\Psi\|_{L^\alpha( U\times (t-1,t) )}\big)^{\mu_2}\\
&\quad \cdot \Big (1+A(\alpha)^\frac {\alpha} {\alpha-a} +\limsup_{t\to\infty} \int_{t-1}^tG_1(\tau)d\tau\Big).
\end{align*}
Then estimate \eqref{LsupGradp-s} follows. 

(ii)  By combining \eqref{Grads0} with \eqref{PbarI} and \eqref{large3}, we have
\begin{align*}
\int_{U'} |\nabla p(x,t)|^s dx 
&\le C \Big\{1+\beta(\alpha)^\frac 1{\alpha-2a} +\sup_{[t-2,t]} A(\alpha,\cdot)^\frac 1{\alpha-a} +  \|\Psi\|_{L^\alpha( U\times (t-2,t) )}\Big\}^{\mu_2}\\
&\quad \cdot \Big (1+\beta(\alpha)^\frac {\alpha}{\alpha-2a} + A(\alpha,t-1)^\frac{\alpha}{\alpha-a} +\int_{t-1}^tG_1(\tau)d\tau\Big ).
\end{align*}
Then \eqref{Gradp-sL} follows.
\end{proof}


\subsection{ $L^\infty$-estimates}\label{subgrad2}
In this subsection, we obtain  interior $L^\infty$-estimates for the gradient of pressure.
For each $m=1,2,\ldots,n,$ denote $u_m=p_{x_m}$ and $u=(u_1,u_2,\ldots,u_n)=\nabla p$. 
We have
\beq\label{dum}
\frac{\partial u_{m}}{\partial t}=\partial_m  ( \nabla \cdot (  K(|u|)  u) )= \nabla \cdot (K(|u|) \partial_m u) +\nabla \cdot \Big[ K'(|u|) \frac{\sum_{i} u_i \partial_m u_i }{|u|} u\Big].          
\eeq
Since $\partial_i u_m=\partial_m u_i$, we have 
$\partial_m u = ( \partial_m u_1,\ldots, \partial_m u_n ) =( \partial_1 u_m, \ldots, \partial_n u_m) =
\nabla u_m$, and
$\sum_{i} u_i \partial_m u_i=\sum_{i} u_i \partial_i u_m=u\cdot \nabla u_m$.
Therefore, we rewrite \eqref{dum} as
\beq \label{um}
\frac{\partial u_{m}}{\partial t} = \nabla\cdot(K(|u|) \nabla u_m) +\nabla \cdot \Big[ K'(|u|) \frac{u \cdot \nabla u_m  }{|u|} u\Big].
\eeq

We will apply De Giorgi's technique to equation \eqref{um}. 
In the following, we fix a number $s_0$ such that $r=s_0$ satisfies \eqref{rcond}. 
Note that $s_0^*>2$.
We will also use $s_j$ for $j\ge 1$ to denote some exponents that depend on $s_0$ but are independent of $\alpha$.
Let 
\beqs 
s_1=(1-2/s_0^*)^{-1}>1.
\eeqs

\begin{theorem}\label{GradUni} Let  $U'  \Subset V\Subset  U$. For any $T_0\ge 0$, $T>0$, and $\theta\in (0,1)$, if $t \in [T_0+\theta T,T_0+T]$ then    
\beq\label{gradInf}
\norm{\nabla p(t)}_{L^\infty(U')}\le C (1+(\theta T)^{-1})^\frac{s_1+1}{2} \lambda^\frac{s_1}2  \norm{\nabla p}_{L^2(V\times (T_0+\theta T/2,T_0+T))}, 
\eeq
where 
\beq\label{lambda}
\lambda=\lambda(T_0, T,\theta;V) = \Big( \int_{T_0+\theta T/2}^{T_0+T} \int_V (1+|\nabla p|)^{\frac {a s_0}{2-s_0}}dx dt  \Big)^\frac{2-s_0}{s_0}.
\eeq
and constant $C>0$ is independent of $T_0$, $T$, and $\theta$.
\end{theorem}
\begin{proof} Without loss of generality, assume $T_0=0$. Fix $m\in\{1,2,\ldots,n\}$. 
We will show for $t\in[\theta T,T]$ that
\beq\label{grad519}
\norm{p_{x_m}(t)}_{L^\infty(U')}\le C(1+(\theta T)^{-1})^{\frac{s_1+1}2}\lambda^{\frac {s_1}2} \|p_{x_m}\|_{L^2(V\times (T_0+\theta T/2,T_0+T))}.
\eeq

 Let $\zeta(x,t)=\phi(x)\varphi(t)$ be a cut-off function with $\varphi(t)=0$ for $t\le \theta T/2$, and supp\,$\phi \subset U$. We define for $k\ge 0$
\begin{align*}
 u_m^{(k)} =\max\{u_m-k,0\},
\quad S_{k}(t)=\{ x\in U: u_m^{(k)}(x,t)\ge 0\},
\end{align*}
and denote by $\chi_k(x,t)$ the characteristic function of $S_k(t)$.

Multiplying \eqref{um} by $u_m^{(k)} \zeta^2$ and integrating over $U$, we have
\begin{multline}\label{extra1}
 \frac 12\frac d{dt}\int_U  |u_m^{(k)}|^2 \zeta^2 dx
=  \int_U |u_m^{(k)}|^2 \zeta|\zeta_t| dx - \int_U  K(|u|)|\nabla u_m^{(k)}|^2 \zeta^2 dx -2\int_U K(|u|)(\nabla u_m^{(k)}\cdot \nabla \zeta)u_m^{(k)} \zeta dx\\ 
 -\int_U \frac1{|u|}K'(|u|) (u\cdot \nabla u_m^{(k)}) (u\cdot \nabla ( u_m^{(k)}\zeta^2 ))  dx.
\end{multline} 
By the product rule,
\beqs
-\frac1{|u|}K'(|u|) (u\cdot \nabla u_m^{(k)})(u\cdot \nabla ( u_m^{(k)}\zeta^2 )) =-\frac1{|u|}K'(|u|) (u\cdot \nabla u_m^{(k)}) \Big( u\cdot \nabla u_m^{(k)} \zeta^2 + 2u_m^{(k)} \zeta u\cdot \nabla \zeta\Big).
\eeqs
It follows property \eqref{K-est-2} that
\beqs
-\frac1{|u|}K'(|u|)(u\cdot \nabla ( u_m^{(k)})^2\zeta^2\le \frac {aK(|u|)}{|u|^2} |u|^2 |\nabla u_m^{(k)}|^2 \zeta^2=aK(|u|) |\nabla u_m^{(k)}|^2 \zeta^2  
\eeqs
and
\beqs
-\frac 2{|u|}K'(|u|)u^2 u_m^{(k)} \zeta  \nabla u_m^{(k)} \cdot   \nabla \zeta \le \frac2{|u|}| K'(|u|)  |u|^2 |\nabla u_m^{(k)}| |u_m^{(k)}|  \zeta |\nabla \zeta| \le C K(|u|) |\nabla u_m^{(k)}| |u_m^{(k)}| \zeta |\nabla \zeta|.
\eeqs
Then we obtain from \eqref{extra1} that  
\begin{align*}
 \frac 12\frac d{dt}\int_U |u_m^{(k)}|^2 \zeta^2 
&\le  \int_U |u_m^{(k)}|^2 \zeta|\zeta_t| dx -(1-a)\int_U  K(|u|)|\nabla u_m^{(k)}|^2 \zeta^2 dx\\
&\quad+ C\int_U K(|u|) |\nabla u_m^{(k)}| |u_m^{(k)}| \zeta |\nabla \zeta|  dx.
\end{align*}
Applying Cauchy's inequality to the last term in previous inequality  yields
\begin{equation*}
 \frac 12\frac d{dt}\int_U |u_m^{(k)}|^2 \zeta^2 
\le  \int_U |u_m^{(k)}|^2 \zeta|\zeta_t| dx -\frac{1-a}{2}\int_U  K(|u|)|\nabla u_m^{(k)}|^2 \zeta^2 dx + C \int_U K(|u|) |u_m^{(k)}|^2 |\nabla \zeta|^2  dx. 
\end{equation*}
Since $|\zeta\nabla u_m^{(k)}|^2=|\nabla (u_m^{(k)} \zeta)-u_m^{(k)}\nabla \zeta|^2\ge \frac12|\nabla (u_m^{(k)} \zeta)|^2 - |u_m^{(k)}\nabla \zeta|^2$, 
we obtain
\beq \label{cut5}
\begin{aligned}
 \frac 12\frac d{dt}\int_U |u_m^{(k)} \zeta|^2 dx 
&+\frac{1-a}4\int_U  K(|u|)|\nabla (u_m^{(k)} \zeta)|^2 dx\\ 
&\le \int_U |u_m^{(k)}|^2 \zeta|\zeta_t| dx + C \int_U K(|u|) |u_m^{(k)}|^2 |\nabla \zeta|^2  dx. 
\end{aligned}
\eeq
Integrating \eqref{cut5} from $0$ to $t$ for $t\in [0,T]$, and then taking supremum in $t$ give 
\beq\label{cut6}
\begin{aligned}
\max_{[0,T]}\int_U |u_m^{(k)} \zeta|^2 dx  
&+C\int_0^T \int_U  K(|u|)|\nabla (u_m^{(k)} \zeta)|^2 dxdt\\
 &\le \int_0^T\int_U |u_m^{(k)}|^2 \zeta|\zeta_t| dx+C \int_0^T \int_U K(|u|) |u_m^{(k)}|^2 |\nabla \zeta|^2  dx dt. 
\end{aligned}
\eeq
Let 
\beq\label{nu2} 
\nu_2=4(1-1/s_0^*)>2.
\eeq
Applying Lemma~\ref{WSobolev} to function $u_m^{(k)} \zeta$ which vanishes on the boundary, weight $W =K(|u|)$ and exponents $r=s_0$, $\varrho=\varrho(s_0)=\nu_2$,  we have 
\begin{align*}
\norm{ u_m^{(k) }\zeta}_{L^{\nu_2}(Q_T)}  &\le C\Big[ \esssup_{t\in [0,T]} \norm{u_m^{(k) }\zeta }_{L^2(U)}+\Big(\int_0^T \int_U K(|u|)|\nabla (u_m^{(k) }\zeta) |^2 dx dt\Big)^\frac 1 2\Big]\\
&\quad\cdot   \Big[\int_0^T\int_{{\rm supp} \zeta} K(|u|)^{-\frac{s_0}{2-s_0}}dx  dt \Big]^{\frac {2- s_0}{\nu_2 s_0} }.
\end{align*}
Using \eqref{Kestn}, we have $K(|u|)^{-\frac{s_0}{2-s_0}}\le C(1+|u|)^{\frac{a s_0}{2-s_0}}$, hence 
\beq\label{cut7}
\norm{ u_m^{(k) }\zeta}_{L^{\nu_2}(Q_T)} \le C\lambda^{1/\nu_2} \Big[ \max_{t\in [0,T]} \int_U |u_m^{(k) }\zeta|^2 dx +\int_0^T \int_U K(|u|)|\nabla (u_m^{(k) }\zeta) |^2 dx dt\Big]^\frac12.
\eeq
By \eqref{cut7}, \eqref{cut6} and the boundedness of function $K(\cdot)$ we find that 
 \beq\label{umk-Ls*}
 \begin{aligned}
\norm{ u_m^{(k) }\zeta}_{L^{\nu_2}(Q_T)} &\le C\lambda^{1/\nu_2} \Big(\max_{[0,T]}\int_U |u_m^{(k)} \zeta|^2 dx  
+C\int_0^T \int_U  K(|u|)|\nabla (u_m^{(k)} \zeta)|^2 dxdt\Big)^\frac12\\
&\le C\lambda^{1/\nu_2} \Big(  \int_0^T\int_U |u_m^{(k)}|^2 \zeta|\zeta_t| dx+\int_0^T \int_U |u_m^{(k)}|^2 |\nabla \zeta|^2  dx dt\Big)^{\frac12}. 
\end{aligned}
\eeq

Here, we use the same notation $x_0$, $\rho$, $M_0$, $k_i$,  $\rho_i$, $U_i$, $\mathcal Q_i$ and $\zeta_i$  as introduced in the proof of Theorem \ref{theo42} from \eqref{indexSeq2} to \eqref{cutoffBound2}. Also, the sets $A_{i,j}$ are defined by \eqref{setDef2} with $p$ being replaced by $u_m$.

Define 
$F_{i} = \norm{u_m^{(k_{i+1})}\zeta_i }_{L^{\nu_2}(A_{i+1, i})}$. 
Applying \eqref{umk-Ls*} with  $k=k_{i+1}$  and $\zeta=\zeta_i$ gives
\beq\label{defFi}
F_i \le C\lambda^{1/\nu_2}  \Big\{ \int_0^T\int_U |u_m^{(k_{i+1})}|^2 \zeta_i|(\zeta_i)_t| dx+ \int_0^T \int_U |u_m^{(k_{i+1})}|^2 |\nabla \zeta_i|^2\Big) dxdt \Big\}^{ 1/2}.
\eeq
Using \eqref{cutoffBound2},  we obtain  
\beqs \label{Fi1}
F_{i}
\le C2^i\lambda^{1/\nu_2} (1+\frac 1{\theta T})^{1/2}   \|u_m^{(k_{i+1})}\|_{L^2(A_{i+1,i})} \le C2^i(1+\frac 1{\theta T})^{1/2}\lambda  \|u_m^{(k_i)}\|_{L^2(A_{i,i})}.
\eeqs
Since $\nu_2>2$, it follows from H\"{o}lder's inequality that
\beq\label{cutum}
\begin{aligned}
 \|u_m^{(k_{i+1})  }\zeta_i \|_{L^2(A_{i+1, i+1})}
& \le \|u_m^{(k_{i+1})  }\zeta_i\|_{L^{\nu_2}(A_{i+1,i+1})} |A_{i+1,i+1}|^{1/2-1/{\nu_2}}\\
& \le \|u_m^{(k_{i+1})  }\zeta_i \|_{L^{\nu_2}(A_{i+1,i+1})} |A_{i+1,i}|^{1/2-1/{\nu_2}}
 \le C  F_{i}  |A_{i+1,i}|^{1/2-1/{\nu_2}}.
\end{aligned}
\eeq
Note that 
$ \|u_m^{(k_i)}\|_{L^2(A_{i,i})}\ge  \|u_m^{(k_{i})}\|_{L^2(A_{i+1,i})}\ge (k_{i+1}-k_i) |A_{i+1,i}|^{1/2}$. 
Thus, 
\beq\label{volA}
  |A_{i+1,i}| \le (k_{i+1}-k_i)^{-2} \|u_m^{(k_i)}\|_{L^2(A_i)}^{2} \le C 4^{i} M_0^{-2}\|u_m^{(k_i)}\|_{L^2(A_{i,i})}^{2}.
\eeq
Then it follows \eqref{cutum}, \eqref{defFi} and \eqref{volA} that
\begin{align*}
\|u_m^{(k_{i+1})}\|_{L^2(A_{i+1,i+1})}&\le C2^i(1+\frac 1{\theta T})^{1/2}\lambda^{1/\nu_2} 
\|u_m^{(k_i)}\|_{L^2(A_i)} 2^{i-\frac {2i} {\nu_2} }M_0^{-1+2/{\nu_2}} \|u_m^{(k_i)}  \|_{L^2(A_{i,i})}^{1-2/{\nu_2}}\\
&\le C 4^i(1+\frac 1{\theta T})^{1/2} \lambda^{1/\nu_2} M_0^{-1+2/{\nu_2}}  \norm{u_m^{(k_i)}}_{L^2(A_{i,i})}^{2-2/{\nu_2}}.
\end{align*}

Denote $\nu_3=1-2/\nu_2$.
Let $Y_i=\| u_m^{(k_i)} \|_{L^2(A_{i,i})}$, $B =4$ and
$D=C(1+\frac 1{\theta T})^{1/2} \lambda^{1/\nu_2} M_0^{-\nu_3}.$
 We obtain 
$$Y_{i+1}\le D B^i Y_i^{1+\nu_3} \quad\text{for all } i\ge 0.$$

We now determine $M_0$ so that $ Y_0 \le D^{-1/\nu_3}B^{-1/{\nu_3^2}}$.
This condition is met if
$$ M_0\ge C  \Big[\lambda^{1/\nu_2} (1+\frac 1{\theta T})^{1/2} \Big]^{1/\nu_3} Y_0=C  \lambda^{s_1/2} (1+\frac 1{\theta T})^\frac{1+s_1}{2} Y_0.$$
Since 
$
Y_0= \norm{u_m^{(k_0)}}_{L^2{(A_{0,0})}}\le \| u_m\|_{L^2(V\times (\theta T/2,T))},
$
 it suffices to choose $M_0$ as
\beq\label{M0} 
M_0 = C \lambda^{s_1/2}(1+\frac 1{\theta T})^\frac{s_1+1}{2} \|u_m\|_{L^2(V\times (\theta T/2,T))}.
\eeq
Then Lemma \ref{multiseq} gives 
$\displaystyle{\lim_{i\to\infty}}Y_i=0$. 
Hence, 
$$\int_{\theta T}^T\int_{B(x_0,\rho/4)} |u_m^{(M_0)}|^2 dxdt=0.$$
Thus, $u_m(x,t)\le M_0$ a.e. in $B(x_0,\rho/4)\times (0,T)$.
Replace $u_m,u$ by $-u_m,-u$ and 
use the same argument we obtain $|u_m(x,t)|\le M_0$ a.e. in $B(x_0,\rho/4)\times (0,T)$.
Now by covering $U'$ by finitely many such balls $B(x_0,\rho/4)$, we come to conclusion
\beq\label{M0um}
|u_m(x,t)|\le M_0 \quad  \text{a.e. in } U'\times (\theta T,T). 
\eeq
By the choice of $M_0$
  we obtain from \eqref{M0um} that
\beqs
|u_m(x,t)|\le C (1+\frac 1{\theta T})^\frac{s_1+1}{2} \lambda^{s_1/2} \|u_m\|_{L^2(V\times (\theta T/2,T))}  
\eeqs
for all $m=1,\ldots,n.$ Then \eqref{grad519} follows.
\end{proof}

We will combine Theorem \ref{GradUni} with the high integrability of $\nabla p$ in subsection \ref{subgrad1} to obtain the $L^\infty$-estimates.  
Let 
\begin{align*}
 s_2&=\max\Big\{2,\frac{as_0}{2-s_0}\Big\},\quad
 s_3=s_1(2-s_0)/s_0+1,\\
\kappa_3&= ( s_2+a-2)\big(1+\frac{2\kappa_1}{\alpha-a}\big),\quad
\kappa_4= 1+s_1+\kappa_3 s_3,\quad 
\kappa_5= \kappa_2( s_2-2+a).
\end{align*}
\begin {theorem}\label{theo57}
If $t\in(0,2)$ then
\begin{multline}\label{Gradps}
\norm{\nabla p(t)}_{L^\infty(U')}\le C t^{-\kappa_4/2} (1+\|\bar p_0\|_{L^\alpha})^{ s_3(\kappa_5+1)}
\Big(1+Env A(\alpha,t)^\frac{1}{\alpha-a} + \|\Psi\|_{L^\alpha(U\times(0,t))}\Big)^{ s_3 \kappa_5}\\
\cdot \Big( 1+\int_0^t G_1(\tau)d\tau\Big)^{ s_3/2}.
\end{multline}
If $t\ge 2$ then
\begin{multline}\label{GradpL}
\norm{\nabla p(t)}_{L^\infty(U')}
\le  C (1+\|\bar p_0\|_{L^\alpha})^{ s_3(\kappa_5+\alpha/2)}  \Big(1+ [ Env A(\alpha,t)]^\frac{1}{\alpha-a}+ \|\Psi\|_{L^\alpha(U\times(t-2,t))}\Big )^{ s_3(\kappa_5+\alpha/2)}\\
\cdot \Big(1+\int_{t-1}^t G_1(\tau)d\tau\Big)^{ s_3/2}.
\end{multline}
\end{theorem}
\begin{proof}
First, we have from \eqref{gradInf} that
\beq\label{gradmain}
\sup_{[T_0+\theta T,T_0+T]}\norm{\nabla p(t)}_{L^\infty(U')}\le C (1+(\theta T)^{-1})^{\frac{1+s_1}{2}} 
\Big( \int_{T_0+\theta T/2}^{T_0+T} \int_V (1+|\nabla p|)^ {s_2} dx dt  \Big)^{ s_3/2}.
\eeq

Let $t\in(0,2)$,  applying \eqref{gradmain} with $T_0=0$, $T=t$, and $\theta=1/2$, we obtain 
\beq\label{gradmain2}
\norm{\nabla p(t)}_{L^\infty(U')}\le C t^{-\frac{1+s_1}{2}} \Big( \int_{t/4}^t \int_V (1+|\nabla p|)^ {s_2} dx dt  \Big)^{ s_3/2}.
\eeq
We apply \eqref{gradps-c} with $s= s_2$ and $U'=V$. Note from formulas in \eqref{mu12} that 
\beq \label{murel}
\mu_1(\alpha, s_2)=\kappa_3+1\quad\text{and}\quad   \mu_2(\alpha, s_2)=2\kappa_5.
\eeq
We obtain
\begin{multline*} \label{gradmain3}
\norm{\nabla p(t)}_{L^\infty(U')}\le C t^{-\frac{1+s_1}{2}}\Big\{ t\cdot t^{-(\kappa_3+1)} 
(1+\|\bar p_0\|_{L^\alpha})^{2(\kappa_5+1)}\\
\cdot \Big(1+Env A(\alpha,t)^\frac{1}{\alpha-a} + \|\Psi\|_{L^\alpha(U\times(0,t))}\Big)^{2\kappa_5}
\Big( 1+\int_0^t G_1(\tau)d\tau\Big)\Big\}^{ s_3/2}\\
\le C t^{-\frac{1+s_1+\kappa_3 s_3}2} 
(1+\|\bar p_0\|_{L^\alpha})^{ s_3 (\kappa_5+1)}\\
\cdot \Big(1+Env A(\alpha,t)^\frac{1}{\alpha-a} + \|\Psi\|_{L^\alpha(U\times(0,t))}\Big)^{ s_3 \kappa_5}
\Big( 1+\int_0^t G_1(\tau)d\tau\Big)^{ s_3/2}.
\end{multline*}
Then \eqref{Gradps} follows.
Let $t\ge 2$. Applying \eqref{gradmain} with $T_0=t-3/4$, $T=3/4$, $\theta=2/3$, then 
 \beq\label{gradlarge}
\norm{\nabla p(t)}_{L^\infty(U')}\le C \Big( 1+\int_{t-1/2}^t \int_V |\nabla p|^{s_2} dx dt  \Big)^{ s_3/2}.
\eeq
Thanks to \eqref{gradps-b} with $s= s_2$, we obtain \eqref{GradpL}.
\end{proof}

Combining \eqref{gradlarge} with Theorem \ref{theo55}, we have the following asymptotic estimates.

\begin{theorem}\label{theo58}
{\rm (i)} If $A(\alpha)<\infty$ then 
\begin{multline}\label{ptwGradp}
\limsup_{t\to\infty}\norm{\nabla p(t)}_{L^\infty(U')}
\le C \Big (1+A(\alpha)^\frac{1}{\alpha-a}+\limsup_{t\to\infty} \|\Psi\|_{L^\alpha( U\times (t-1,t))}\Big)^{{ s_3}(\kappa_5+\alpha/2)}\\
\cdot\Big(1+ \limsup_{t\to\infty} \int_{t-1}^tG_1(\tau)d\tau\Big)^{ s_3/2}.
\end{multline}

{\rm (ii)} If $\beta(\alpha)<\infty$  then there is $T>0$ such that when $t>T$ we have
\begin{multline}\label{ptwGradpL}
\norm{\nabla p(t)}_{L^\infty(U')}
\le C \Big( 1 + \beta(\alpha)^\frac{1}{\alpha-2a} + \sup_{[t-3,t]}A(\alpha,\cdot)^\frac{1}{\alpha-a} 
 + \|\Psi\|_{L^\alpha( U\times (t-3,t) )}\Big)^{{ s_3}(\kappa_5+\alpha/2)}\\
\cdot \Big(1+\int_{t-2}^t G_1(\tau)d\tau\Big)^{ s_3/2}.
\end{multline}
\end{theorem}

\myclearpage

\section{Interior estimates for time derivative of pressure}
\label{Lpt-sec}
In this section, we estimate the $L^\infty$-norm of $p_t(x,t)$ for $t>0$.
Let $q=p_t$. Then
\beq\label{eqt} \frac{\partial q}{\partial t}=\nabla \cdot \big(K(|\nabla p|)\nabla p\big)_t.\eeq 

Using \eqref{eqt}, we first derive a local-in-time estimate for $L^\infty$-norm of $p_t$.

\begin{proposition} \label{ptInf} 
Let  $U'  \Subset V\Subset  U$. If $T_0\ge 0$, $T>0$ and $\theta\in(0,1)$, then
\beq\label{ptbound}
\sup_{[T_0+\theta T,T_0+T]}\| p_t\|_{L^\infty(U')}\le C\lambda^\frac{s_1}{2} (1+(\theta T)^{-1} )^\frac{s_1+1}{2} \| p_t\|_{L^2(U\times (T_0,T_0+T))},
\eeq
where $s_1$ and $\lambda=\lambda(T_0,T,\theta;V)$ are defined in Theorem \ref{GradUni},
and constant $C>0$ is independent of $T_0$, $T$, and $\theta$.
\end{proposition}

\begin{proof} Without loss of generality, assume $T_0=0$.
For $k\ge 0$, let $ q^{(k)}=\max\{ q-k,0\}$
and $ S_k(t)=\{x \in U: q(x,t)>k\}$,
and $\chi_k(x,t)$  be the characteristic function of set $\{(x,t) \in U\times(0,T):  q(x,t)>k\}$.
On $S_k(t)$,  we have $(\nabla p)_t=\nabla q=\nabla q^{(k)}$. 

Let $\zeta=\zeta(x,t)$ be the cut-off function on $U\times [0,T]$ satisfying $\zeta(\cdot,0)=0$ and  $\zeta(\cdot,t)$ having compact support in $U$. 
We will use test function $ q^{(k)}\zeta^2$, noting that $\nabla ( q^{(k)}\zeta^2) =\zeta [\nabla ( q^{(k)}\zeta)+ q^{(k)}\nabla \zeta]$. Multiplying \eqref{eqt} by $ q^{(k)}\zeta^2 $ and integrating the resultant on $U$, we get 
\beq\label{newadd}
\begin{aligned}
\frac 12\ddt \int_U |q^{(k)}\zeta|^2 dx & 
= \int_U |q^{(k)}|^2 \zeta \zeta_t  dx - \int_U (K(|\nabla p|))_t\nabla p \cdot[ \nabla (q^{(k)}\zeta)+ q^{(k)}\nabla \zeta]\zeta  dx \\
& \quad 
- \int_U K(|\nabla p|) (\nabla p)_t \cdot [ \nabla (q^{(k)}\zeta)+ q^{(k)}\nabla\zeta]\zeta  dx. 
\end{aligned}
\eeq

For the last integral of \eqref{newadd}, put $z=\zeta[\nabla (q^{(k)}\zeta)+ q^{(k)}\nabla \zeta]$.
We have
\beqs
(\nabla p)_t \cdot z 
=\zeta \nabla q^{(k)}\cdot [\nabla (q^{(k)}\zeta)+ q^{(k)}\nabla \zeta]
= |\nabla (q^{(k)}\zeta)|^2 - |q^{(k)} \nabla \zeta|^2.
\eeqs

For the second integral on the right-hand side of \eqref{newadd}, taking into account \eqref{K-est-2},   
\begin{align*}
|(K(|\nabla p|))_t\nabla p \cdot z|\
&=|K'(|\nabla p|)| \frac{|\nabla p\cdot \nabla p_t|}{|\nabla p|} |\nabla p \cdot z| 
\le a K(|\nabla p|) |\nabla q||z|.
\end{align*}
Note that  
\begin{align*}
|\nabla q| |z|
&=|\zeta \nabla\bar  q^{(k)}| | \nabla (q^{(k)}\zeta)+ q^{(k)}\nabla \zeta |\le \big\{|\nabla (q^{(k)}\zeta)| + |q^{(k)}||\nabla \zeta|\big\}^2 \\
&= |\nabla (q^{(k)}\zeta)|^2 + 2|q^{(k)}| |\nabla \zeta| |\nabla (q^{(k)}\zeta)| + |q^{(k)} \nabla \zeta|^2,
\end{align*}
and by applying Cauchy's inequality to the second to last term
\begin{align*}
a|\nabla q| |z|
\le a |\nabla (q^{(k)}\zeta)|^2 + \frac {1-a}{2} |\nabla(q^{(k)}\zeta )|^2 + \frac{2a^2}{1-a} |\bar  q^{(k)} \nabla \zeta|^2 + a|q^{(k)} \nabla \zeta|^2.
\end{align*}

It follows the above calculations that
\begin{align*}
\ddt \int_U |q^{(k)} \zeta|^2 dx &+ (1-a) \int_U K(|\nabla p|) |\nabla (q^{(k)} \zeta)|^2 dx \\
&\quad  \le 2\int_U |q^{(k)}|^2 \zeta  |\zeta_t| dx + C \int_U K(|\nabla p|)|q^{(k)} \nabla \zeta|^2 dx.
\end{align*}
Integrating this inequality from $0$ to $T$, we obtain 
\begin{align*}
 \max_{[0,T]} \int_U |q^{(k)}\zeta |^2 dx &+ \int_0^T \int_U K(|\nabla p|) |\nabla (q^{(k)}\zeta) |^2 dx dt\\
 &\le C \Big [\int_0^T \int_U |q^{(k)}|^2 \zeta |\zeta_t| dx dt 
 + \int_0^T\int_U K(|\nabla p|) |q^{(k)} \nabla \zeta|^2 dxdt   \Big]\\
& \le C\int_0^T  \int_U |q^{(k)}|^2 (\zeta |\zeta_t|  +|\nabla \zeta|^2) dxdt .
 \end{align*}
 The last inequality uses the fact that function $K(\cdot)$ is bounded above.
 Applying Lemma \ref{WSobolev} to $q^{(k)} \zeta$ with $W=K(|\nabla p|)$ and $\varrho=\nu_2$ defined by \eqref{nu2}, we have 
  \begin{align*}
&   \norm{q^{(k)}\zeta}_{L^{\nu_2}(Q_T)}
 \le C \lambda^{1/\nu_2} 
  \cdot \left\{\max_{[0,T]} \int_U |q^{(k)}\zeta |^2 dx + \int_0^T \int_U K(|\nabla p|)|\nabla (q^{(k)}\zeta)|^2 dx dt \right\}^{1/2}.
 \end{align*} 
 and hence, 
  \beq\label{qzetabound}
  \begin{split}
  \norm{q^{(k)}\zeta}_{L^{\nu_2}(Q_T)} &\le C\lambda^{1/\nu_2}  \Big\{\int_0^T  \int_U  |q^{(k)}|^2 (|\zeta_t|\zeta + |\nabla \zeta|^2)dx dt \Big\}^{1/2}.
   \end{split}
   \eeq
This is similar to inequality \eqref{umk-Ls*}. Then by following arguments of Theorem~\ref{GradUni} applied for $q^{(k)}$ instead of $u_m^{(k)}$, we obtain \eqref{ptbound}. The proof is complete. 
\end{proof}

The next theorem contains the estimates in terms of the initial and boundary data for all $t>0$.
Let 
\begin{align*}
\kappa_6&= 1+s_1+\kappa_3(s_3-1),\quad \kappa_7= (s_3-1)(\kappa_5+1)+1,\\
\kappa_8&= (s_3-1) \kappa_5,\quad \kappa_9= (s_3-1)(\kappa_5+\alpha/2)+\alpha/2.
\end{align*}

\begin{theorem}\label{ptthm} 
Let $U'\Subset U$. For $t\in(0,2)$,
\beq\label{i66}
\begin{aligned}
 \| p_t(t)\|_{L^\infty(U')}
  &\le C  t^{-\kappa_6/2} \Big(1+\|\bar{p}_0\|_{L^\alpha}\Big)^{\kappa_7} \Big(1+ \int_U H(|\nabla p_0(x)|)dx \Big)^{1/2}.\\
&\quad\cdot \Big(1+ [ Env A(\alpha,t)]^\frac{1}{\alpha-a}+ \|\Psi\|_{L^\alpha(U\times(0,t))}\Big )^{\kappa_8}
\Big(1+\int_0^t G_3(\tau)d\tau)\Big)^{ s_3/2}.
\end{aligned}
\eeq
For $t\ge 2$,
\beq\label{i67}
\begin{aligned}
\| p_t(t)\|_{L^\infty(U')}
&\le C (1+\|\bar p_0\|_{L^\alpha})^{\kappa_9}  \Big(1+ [ Env A(\alpha,t)]^\frac{1}{\alpha-a}+ \|\Psi\|_{L^\alpha(U\times(t-2,t))}\Big )^{\kappa_9}\\
&\quad \cdot \Big( 1+\int_{t-1}^t G_3(\tau)d\tau\Big)^{ s_3/2}.
\end{aligned}
\eeq
\end{theorem}
\begin{proof}
Let $t\in(0,2)$, apply \eqref{ptbound} with $T_0=0$, $T=t$ and $\theta=1/2$. By \eqref{lambda} and \eqref{new0}, we have
\begin{align*}
\lambda^{s_1/2} & = \Big( \int_{t/2}^{t} \int_V (1+|\nabla p|)^{\frac {a s_0}{2-s_0}}dx d\tau  \Big)^\frac{(2-s_0)s_1}{2 s_0}
\le \Big( \int_{t/2}^{t} \int_V (1+|\nabla p|)^{s_2} dx d\tau  \Big)^\frac{\nu_4}{2},
\end{align*}
where
\beq \label{nu4}
\nu_4=(2-s_0)s_1/s_0= s_3-1.
\eeq
By \eqref{gradps-c} and relation \eqref{murel},
\begin{align*}
\lambda^{s_1/2}&\le C \Big\{ t\cdot t^{-(\kappa_3+1)}(1+\|\bar p_0\|_{L^\alpha})^{2(\kappa_5+1)} \Big(1+ [ Env A(\alpha,t)]^\frac{1}{\alpha-a}+ \|\Psi\|_{L^\alpha(U\times(0,t))}\Big )^{2\kappa_5}\\
&\quad\quad \cdot\Big( 1+\int_0^t G_1(\tau)d\tau\Big)\Big\}^{\nu_4/2}
\le C t^\frac{-\kappa_3\nu_4}{2}(1+\|\bar p_0\|_{L^\alpha})^{(\kappa_5+1)\nu_4}\\
&\quad\quad \cdot  \Big(1+ [ Env A(\alpha,t)]^\frac{1}{\alpha-a}+ \|\Psi\|_{L^\alpha(U\times(0,t))}\Big )^{\kappa_5\nu_4}
\Big( 1+\int_0^t G_1(\tau)d\tau\Big)^{\nu_4/2}.
\end{align*}
Combining this with \eqref{ptbound} and \eqref{intpt} gives
\begin{align*}
& \| p_t(t)\|_{L^\infty(U')}
\le C\lambda^\frac{s_1}{2}  t^{-\frac{1+s_1}{2}} \| p_t\|_{L^2(U\times (0,t))}\\
 &\le C t^\frac{-\kappa_3\nu_4-(1+s_1)}{2}(1+\|\bar p_0\|_{L^\alpha})^{(\kappa_5+1)\nu_4} \Big(1+ [ Env A(\alpha,t)]^\frac{1}{\alpha-a}+ \|\Psi\|_{L^\alpha(U\times(0,t))}\Big )^{\kappa_5\nu_4}\\
&\quad \cdot\Big( 1+\int_0^t G_1(\tau)d\tau\Big)^{\nu_4/2}\\
&\quad \cdot\Big( \int_U [H(|\nabla p_0(x)|)+\bar p^2_0(x)] dx +\int_0^t G_3(\tau)d\tau +\int_0^t \int_U |\Psi_t(x,\tau)|^2dxd\tau\Big)^{1/2}.
\end{align*}
Therefore,
\begin{align*}
 \| p_t(t)\|_{L^\infty(U')}
  &\le C  t^{-\kappa_6/2} \Big(1+\|\bar{p}_0\|_{L^\alpha}\Big)^{(\kappa_5+1)\nu_4+1} \Big(1+ \int_U H(|\nabla p_0(x)|)dx \Big)^{1/2}.\\
&\quad\cdot \Big(1+ [ Env A(\alpha,t)]^\frac{1}{\alpha-a}+ \|\Psi\|_{L^\alpha(U\times(0,t))}\Big )^{\kappa_5\nu_4}
\Big(1+\int_0^t G_3(\tau)d\tau)\Big)^{(\nu_4+1)/2}.
\end{align*}
We obtain \eqref{i66}.

Now consider $t\ge 2$. We apply \eqref{ptbound} with $T_0=t-1$, $T=1$ and $\theta=1/2$.
Then by \eqref{lambda} and \eqref{gradps-b},
\begin{align*}
\lambda^{s_1/2}
&\le C \Big( \int_{t-1}^{t} \int_V (1+|\nabla p|)^ {s_2} dx d\tau  \Big)^{\nu_4/2}\\
&\le C (1+\|\bar p_0\|_{L^\alpha})^{\nu_4(\kappa_5+\alpha/2)} \Big(1+ [ Env A(\alpha,t)]^\frac{1}{\alpha-a}+ \|\Psi\|_{L^\alpha(U\times(t-2,t))}\Big )^{\nu_4(\kappa_5+\alpha/2)}\\
&\quad \cdot \Big( 1+\int_{t-1}^t G_1(\tau)d\tau\Big)^{\nu_4/2}.
\end{align*}
Combining this with \eqref{ptbound} and \eqref{all2} gives
\begin{align*}
 &\| p_t(t)\|_{L^\infty(U')}
 \le C \lambda^\frac{s_1}{2} \| p_t\|_{L^2(U\times (t-1/2,t))}\\
&\le C (1+\|\bar p_0\|_{L^\alpha})^{\nu_4(\kappa_5+\alpha/2)}  \Big(1+ [ Env A(\alpha,t)]^\frac{1}{\alpha-a}+ \|\Psi\|_{L^\alpha(U\times(t-2,t))}\Big )^{\nu_4 (\kappa_5+\alpha/2)}\\
&\quad \cdot \Big( 1+\int_{t-1}^t G_1(\tau)d\tau\Big)^{\nu_4/2} \cdot \Big(1+ \int_U|\bar{p}_0(x)|^{\alpha} dx
+ [ Env A(\alpha,t)]^\frac{\alpha}{\alpha-a} + \int_{t-1}^tG_3(\tau)d\tau\Big )^{1/2}.
\end{align*}
Thus, 
\begin{multline*}
\| p_t(t)\|_{L^\infty(U')}
\le C (1+\|\bar p_0\|_{L^\alpha})^{\nu_4(\kappa_5+\alpha/2)+\alpha/2}\\
\cdot  \Big(1+ [ Env A(\alpha,t)]^\frac{1}{\alpha-a}+ \|\Psi\|_{L^\alpha(U\times(t-2,t))}\Big )^{\nu_4(\kappa_5+\alpha/2)+\alpha/2}
 \Big( 1+\int_{t-1}^t G_3(\tau)d\tau\Big)^{\nu_4/2+1/2},
\end{multline*}
and we obtain \eqref{i67}. The proof is complete.
\end{proof}

For large time or asymptotic estimates, we have:

\begin{theorem}\label{ptlimthm}
Let $U'\Subset U$. 

 {\rm (i)} If $ A(\alpha)<\infty$ then
\beq\label{lim7} 
\begin{aligned}
\limsup_{t\to\infty} \| p_t(t)\|_{L^\infty(U')}
& \le C \Big (1+A(\alpha)^\frac{1}{\alpha-a}+\limsup_{t\to\infty} \|\Psi\|_{L^\alpha( U\times (t-1,t))}\Big)^{\kappa_9}\\
&\quad\cdot \Big(1+ \limsup_{t\to\infty} \int_{t-1}^tG_3(\tau)d\tau\Big)^{ s_3/2}.\end{aligned}
\eeq
 
{\rm (ii)} If $\beta(\alpha)<\infty$ then there is $T>0$ such that for all $t>T$,
 \beq\label{large7}
\begin{aligned}
  \| p_t(t)\|_{L^\infty(U')}
 &\le C \Big( 1 + \beta(\alpha)^\frac{1}{\alpha-2a} + \sup_{[t-3,t]}A(\alpha,\cdot)^\frac{1}{\alpha-a} 
 + \|\Psi\|_{L^\alpha( U\times (t-3,t) )}\Big)^{\kappa_9}\\
&\quad \cdot\Big(   1+\int_{t-2}^tG_3(\tau)d\tau\Big)^{ s_3/2}.
\end{aligned}
\eeq
 \end{theorem}
\begin{proof}
(i)  For large $t$, we apply \eqref{ptbound} with $T_0=t-1$, $T=1$ and $\theta=1/2$. We have
 \beq\label{ptlarge}
 \| p_t(t)\|_{L^\infty(U')}
 \le C \lambda^{s_1/2} \| p_t\|_{L^2(U\times (t-1/2,t))}
\le C \Big( \int_{t-1}^{t} \int_V (1+|\nabla p|)^{s_2} dx d\tau  \Big)^{\nu_4/2} \| p_t\|_{L^2(U\times (t-1/2,t))},
\eeq
where $\nu_4$ is defined by \eqref{nu4}. Take the limit superior and using \eqref{LsupGradp-s} with $s= s_2$, and \eqref{lim4}, we obtain
 \begin{align*}
&\limsup_{t\to\infty} \| p_t(t)\|_{L^\infty(U')}\\
& \le C \Big (1+A(\alpha)^\frac{1}{\alpha-a}+\limsup_{t\to\infty} \|\Psi\|_{L^\alpha( U\times (t-1,t))}\Big)^{\nu_4(\kappa_5+\alpha/2)}
\Big(1+ \limsup_{t\to\infty} \int_{t-1}^tG_1(\tau)d\tau\Big)^{\nu_4/2}\\
&\quad \cdot \Big (1+A(\alpha)^\frac {\alpha} {\alpha-a} +\limsup_{t\to\infty}\int_{t-1}^t \int_U |\Psi_t(x,\tau)|^2dxd\tau+\limsup_{t\to\infty} \int_{t-1}^tG_3(\tau)d\tau\Big)^{1/2}\\
& \le C \Big (1+A(\alpha)^\frac{1}{\alpha-a}+\limsup_{t\to\infty} \|\Psi\|_{L^\alpha( U\times (t-1,t))}\Big)^{\nu_4(\kappa_5+\alpha/2)+\alpha/2}
\Big(1+ \limsup_{t\to\infty} \int_{t-1}^tG_3(\tau)d\tau\Big)^{\nu_4/2+1/2}.
\end{align*}
We obtain \eqref{lim7}.

(ii) Using \eqref{ptlarge}, \eqref{Gradp-sL},  and \eqref{large4}, we obtain
 \begin{align*}
& \| p_t(t)\|_{L^\infty(U')}
 \le C \Big( 1 + \beta(\alpha)^\frac{1}{\alpha-2a} + \sup_{[t-3,t]}A(\alpha,\cdot)^\frac{1}{\alpha-a} 
 + \|\Psi\|_{L^\alpha( U\times (t-3,t) )}\Big)^{\nu_4(\kappa_5+\alpha/2)}
\Big(1+\int_{t-2}^t G_1(\tau)d\tau\Big)^{\nu_4/2}\\
&\quad \cdot\Big (1+\beta(\alpha)^\frac {\alpha}{\alpha-2a} + A(\alpha,t-1)^\frac{\alpha}{\alpha-a}+\int_{t-1}^t \int_U |\Psi_t(x,\tau)|^2dxd\tau +\int_{t-1}^tG_3(\tau)d\tau\Big )^{1/2}.
\end{align*}
Therefore \eqref{large7} follows.
\end{proof}

\myclearpage
\section{Interior estimates for pressure's Hessian}
\label{Hess-sec}
In this section, we estimate the $L^2$-norm of the Hessian $\nabla^2 p=(p_{x_ix_j})_{i,j=1,2,\ldots,n}$.

\begin{lemma}\label{HessLem} Let $U'\Subset V \Subset U$.
For $t>0$,
\beq\label{mainHess}
\norm{\nabla^2 p(t)}_{L^2(U')}\le C 
 \Big(1+\norm{\nabla p(t)}_{L^\infty(V) }\Big)^a \Big(\int_U [|\nabla p|^{2-a}+|p_t|^2]dx\Big)^{1/2}.
\eeq
\end{lemma}
\begin{proof}  
From \eqref{gradid} of Lemma~\ref{new.it-L} with  $\zeta=\zeta(x)$ being a cut-off function in space, we have
\beq\label{Hess}
\begin{aligned}
  \frac{1}{2s+2} \ddt \int_U |\nabla p|^{2s+2} \zeta^2 dx
+\frac{1-a}{2}  \int_U K(|\nabla p|) |\nabla^2 p|^2   |\nabla p|^{2s}  \zeta^2 dx\\
 \le  C \int_U K(|\nabla p|) |\nabla p|^{2s+2}|\nabla \zeta|^2 dx.
\end{aligned}
\eeq
Clearly, 
\beq\label{s0}
\frac{1}{2s+2} \ddt \int_U |\nabla p|^{2s+2} \zeta^2 dx= - \int_U  p_t  \nabla \cdot (|\nabla p|^{2s}\nabla p \zeta^2)  dx.
\eeq
Combining \eqref{Hess} and \eqref{s0} with $s=0$, we have 
\begin{align*}
 \int_U K(|\nabla p|) |\nabla^2 p|^2   \zeta^2 dx\le  C \int_U K(|\nabla p|) |\nabla p|^{2}|\nabla \zeta|^2 dx
+ C   \int_U | p_t \nabla \cdot (\nabla p \zeta^2)|  dx.
\end{align*}
Since
\begin{align*}
 | p_t \nabla \cdot (\nabla p \zeta^2)| 
&\le  | p_t| |\nabla ^2  p| \zeta^2 + 2| p_t||\nabla p| \zeta |\nabla \zeta| \\
&\le 1/2  K(|\nabla p|) |\nabla^2 p|^2     \zeta^2 
+C K(|\nabla p|) |\nabla p|^{2}|\nabla \zeta|^2 
+ C| p_t|^2  K^{-1}(|\nabla p|) \zeta^2, 
\end{align*}
and, by \eqref{Kestn}, $K^{-1}(\xi)\le C (1+\xi)^a$ we find that
\beq\label{Hess1}
\begin{aligned}
\int_U K(|\nabla p|) |\nabla^2 p|^2   \zeta^2 dx
\le   C \int_U K(|\nabla p|) |\nabla p|^{2}|\nabla \zeta|^2 dx+C\int_U | p_t|^2 (1+|\nabla   p|)^a \zeta^2 dx.
\end{aligned}
\eeq
Constructing appropriate $\zeta$ in \eqref{Hess1} with $\zeta\equiv 1$ on $U'$ and supp $\zeta \subset V$ we obtain 
\beq\label{Hess2}
  \int_{U'} K(|\nabla p|) |\nabla^2 p|^2 dx\le C  (1+  \| \nabla p\|_{L^\infty(V)})^a \int_{ V }[|\nabla p|^{2-a}+|p_t|^2]dx  .
\eeq
For $x\in U'$,
\beq\label{Hess3}
K(|\nabla p|) \ge  C (1+|\nabla p|)^{-a}\ge  C ((1+\norm{\nabla p}_{L^\infty(V)})^{-a}
\eeq
 From \eqref{Hess2}, \eqref{Hess3} and Young's inequality we obtain \eqref{mainHess}.     
\end{proof}

Now, we combine Lemma \ref{HessLem} with estimates in section \ref{revision} and subsection \ref{subgrad2} to obtain particular bounds for the Hessian.
Let 
\begin{align*}
s_4&=as_3+1,\\
\kappa_{10}&= a\kappa_4+1,\quad
\kappa_{11}= a s_3(\kappa_5+\alpha/2)+\alpha/2,\quad   
\kappa_{12}= a s_3 \kappa_5+\alpha/2.
\end{align*}
\begin{theorem}\label{theo72} 
Let $U'\Subset U$. 

{\rm (i)} If $t\in(0,2)$ then
\begin{multline}\label{small8}
\norm{\nabla^2 p(t) }_{L^2(U')}
\le C t^{-\kappa_{10}/2} (1+\|\bar p_0\|_{L^\alpha})^{\kappa_{11}} 
\Big(1+\int_U H(|\nabla p_0(x)|) dx\Big)^{1/2}\\
\cdot \Big(1+Env A(\alpha,t)^\frac{1}{\alpha-a} + \|\Psi\|_{L^\alpha(U\times(0,t))}\Big)^{\kappa_{12}}
 \Big( 1+\int_0^t G_4(\tau)d\tau\Big)^{s_4/2}.
\end{multline}

{\rm (ii)} If $t\ge 2$ then 
\begin{multline}\label{t8}
\norm{\nabla^2 p(t) }_{L^2(U')}
\le C (1+\|\bar p_0\|_{L^\alpha})^{\kappa_{11}}  \Big(1+ [ Env A(\alpha,t)]^\frac{1}{\alpha-a}+ \|\Psi\|_{L^\alpha(U\times(t-2,t))}\Big )^{\kappa_{11}}\\
\cdot \Big(1+\int_{t-1}^t G_4(\tau)d\tau\Big)^{s_4/2}.
\end{multline}

{\rm (iii)} If $A(\alpha)<\infty$ then 
\begin{multline}\label{lim8}
\limsup_{t\to\infty} \norm{\nabla^2 p }_{L^2(U')}
\le C \Big (1+A(\alpha)^\frac{1}{\alpha-a}+\limsup_{t\to\infty} \|\Psi\|_{L^\alpha( U\times (t-1,t))}\Big)^{\kappa_{11}}\\
\cdot\Big(1+ \limsup_{t\to\infty} \int_{t-1}^tG_4(\tau)d\tau\Big)^{s_4/2}. 
\end{multline}

{\rm (iv)} If $\beta(\alpha)<\infty$ then there is $T>0$ such that for $t>T$,
\begin{multline}\label{large8}
\norm{\nabla^2 p(t) }_{L^2(U')}
\le C \Big( 1 + \beta(\alpha)^\frac{1}{\alpha-2a} + \sup_{[t-3,t]}A(\alpha,\cdot)^\frac{1}{\alpha-a} 
 + \|\Psi\|_{L^\alpha( U\times (t-3,t) )}\Big)^{\kappa_{11}}\\
\cdot \Big(1+\int_{t-2}^t G_4(\tau)d\tau\Big)^{s_4/2}.
\end{multline}
\end{theorem}
\begin{proof}
(i) Using \eqref {mainHess},  \eqref{Gradps}, \eqref{intpt} and \eqref{newHpt2}:
\begin{multline*}
\norm{\nabla^2 p(t) }_{L^2(U')}
\le C t^{-a\kappa_4/2} (1+\|\bar p_0\|_{L^\alpha})^{a s_3(\kappa_5+1)}
\Big(1+Env A(\alpha,t)^\frac{1}{\alpha-a} + \|\Psi\|_{L^\alpha(U\times(0,t))}\Big)^{a s_3 \kappa_5}\\
\cdot \Big( 1+\int_0^t G_1(\tau)d\tau\Big)^{a s_3/2}
\Big\{ t^{-1} \Big(1+\int_U |\bar p_0(x)|^\alpha dx+ \int_U H(|\nabla p_0(x)|) dx \\
+ [Env A(\alpha,t)]^\frac {\alpha}{\alpha-a}+\int_U |\Psi_t(x,t)|^2dx+  \int_{0}^t G_4(\tau)  d\tau\Big)\Big\}^{1/2}.
\end{multline*}
Note that  $\alpha\ge 2$, and 
\beqs
\int_U |\Psi_t(x,t)|^2dx \le C\Big(\int_U |\Psi_t(x,t)|^\alpha dx\Big)^\frac{2}{\alpha}
\le C A(\alpha,t)^\frac{2(1-a)}{\alpha-a} \le C(1+A(\alpha,t)^\frac{\alpha}{\alpha-a}).
\eeqs
Then \eqref{small8} follows.

(ii) Using \eqref {mainHess},  \eqref{GradpL}, \eqref{t6}, \eqref{ptUni} and \eqref{pbar:ineq1}:
\begin{align*}
&\norm{\nabla^2 p(t) }_{L^2(U')}
\le C (1+\|\bar p_0\|_{L^\alpha})^{a s_3(\kappa_5+\alpha/2)}  \Big(1+ [ Env A(\alpha,t)]^\frac{1}{\alpha-a}+ \|\Psi\|_{L^\alpha(U\times(t-2,t))}\Big )^{a s_3(\kappa_5+\alpha/2)}\\
&\quad \cdot \Big(1+\int_{t-1}^t G_1(\tau)d\tau\Big)^{a s_3/2}\\
&\quad \cdot \Big(1+ \int_U|\bar{p}_0(x)|^{\alpha} dx
+ [ Env A(\alpha,t-1)]^\frac{\alpha}{\alpha-a} +\int_U |\Psi_t(x,t-1)|^2dx+\int_{t-1}^tG_4(\tau)d\tau\Big )^{1/2} .
\end{align*}
Then \eqref{t8} follows.

(iii) Using \eqref {mainHess},  \eqref{ptwGradp}, \eqref{lim1} and \eqref{lim2}:
\begin{multline*}
\limsup_{t\to\infty} \norm{\nabla^2 p }_{L^2(U')}
\le C \Big (1+A(\alpha)^\frac{1}{\alpha-a}+\limsup_{t\to\infty} \|\Psi\|_{L^\alpha( U\times (t-1,t))}\Big)^{a  s_3(\kappa_5+\alpha/2)}\\
\cdot\Big(1+ \limsup_{t\to\infty} \int_{t-1}^tG_1(\tau)d\tau\Big)^{a s_3/2}
 \Big (1+A(\alpha)^\frac {\alpha} {\alpha-a} +\limsup_{t\to\infty} \int_U |\Psi_t(x,t)|^2 dx +\limsup_{t\to\infty} \int_{t-1}^tG_4(\tau)d\tau\Big)^{1/2} .
\end{multline*}
 Then \eqref{lim8} follows.

(iv) Using \eqref {mainHess}, \eqref{ptwGradpL}, \eqref{large1} and \eqref{large2}:
\begin{multline*}
\norm{\nabla^2 p(t) }_{L^2(U')}
\le C \Big( 1 + \beta(\alpha)^\frac{1}{\alpha-2a} + \sup_{[t-3,t]}A(\alpha,\cdot)^\frac{1}{\alpha-a} 
 + \|\Psi\|_{L^\alpha( U\times (t-3,t) )}\Big)^{a  s_3(\kappa_5+\alpha/2)}\\
\cdot \Big(1+\int_{t-2}^t G_1(\tau)d\tau\Big)^{a s_3/2}
\Big (1+\beta(\alpha)^\frac {\alpha}{\alpha-2a} + A(\alpha,t-1)^\frac{\alpha}{\alpha-a} +\int_{t-1}^tG_4(\tau)d\tau +\int_U |\Psi_t(x,t)|^2 dx\Big )^{1/2} .
\end{multline*}
Then \eqref{large8} follows.
\end{proof}

\appendix
\myclearpage
 \section{Appendix}\label{ElePfs}
\begin{proof}[Proof of Lemma \ref{Sob4}]
In this proof, we use notation $\|\cdot\|_{L^p}$ to denote $\|\cdot\|_{L^p(U)}$.
First, we have the following Poincar\'e-Sobolev inequality
\beq\label{L1S}
\| \phi\|_{L^{r^*}}\le C \|\nabla \phi\|_{L^r}+C\delta\| \phi\|_{L^1}.
\eeq
Let $r=2-a$ and $\phi=|u|^m$, where $m=(\alpha-a)/(2-a)\ge 1$. 
Applying \eqref{L1S} to $\phi$,  we obtain
\beq\label{pre1}
\|u\|_{L^{m r^*}}\le C\Big( \int_U |u|^{\alpha-2 }|\nabla u|^{2-a} dx\Big)^\frac{1}{\alpha-a} +C\delta\| u\|_{L^m}
\le C\Big (\int_U |u|^{\alpha-2 }|\nabla u|^{2-a} dx\Big)^\frac{1}{\alpha-a} +C\delta\| u\|_{L^\alpha}. 
\eeq
The last inequality comes from H\"older's inequality and the fact $\alpha > m$.
Consider a number $p\in (\alpha,mr^*)$. Let $\theta\in(0,1)$ such that
\beq\label{psplit}
\frac 1 p =\frac\theta \alpha + \frac{1-\theta}{mr^*}.
\eeq
Combining interpolation inequality and \eqref{pre1} yields
\begin{align*}
\|u\|_{L^p}^p
&\le \|u\|_{L^\alpha}^{\theta p} \|u\|_{L^{mr^*}}^{(1-\theta)p}
\le C \|u\|_{L^\alpha}^{\theta p} \Big\{ \Big (\int_U |u|^{\alpha-2 }|\nabla u|^{2-a} dx\Big)^\frac{1}{\alpha-a} +C\delta\| u\|_{L^\alpha}\Big\}^{(1-\theta)p}\\
&\le C \delta\|u\|_{L^\alpha}^{p} + \|u\|_{L^\alpha}^{\theta p} \Big (\int_U |u|^{\alpha-2 }|\nabla u|^{2-a} dx\Big)^\frac{(1-\theta)p}{\alpha-a}.
\end{align*}
In the preceding inequality, selecting $(1-\theta)p=\alpha-a$, letting $u=u(x,t)$, and integrating in $t$ from $0$ to $T$  give
\begin{align*}
\int_0^T \|u(t)\|_{L^p}^pdt
&\le C \delta T \esssup_{[0,T]}  \|u\|_{L^\alpha}^{p} + C \esssup_{[0,T]} \|u\|_{L^\alpha}^{\theta p} \int_0^T \int_U |u|^{\alpha-2 }|\nabla u|^{2-a} dx dt\\
& \le C\delta T [[u]]^p + C [[u]]^{\theta p} [[u]]^{(1-\theta)p}=C(1+\delta T)[[u]]^p.
\end{align*}
Therefore \eqref{nonzerobdn} follows. 
It remains to calculate $p$ and check sufficient conditions. 
Note that $\theta p=p-\alpha+a$, and from \eqref{psplit}, we derive
$$ 1=\frac{\theta p}{\alpha}+\frac{(1-\theta)p}{mr^*} =\frac{p-\alpha+a}{\alpha}+\frac{\alpha-a}{mr^*}.$$
Solving for $p$ from this gives
$ p=\alpha +(\alpha-a)(1-\frac{\alpha}{mr^*})$.
Simple calculations yield $p$ in the form of \eqref{expndef}.
It is elementary to verify that the second condition in \eqref{alcond} guarantees $\alpha<p<mr^*$. 

In case $U$ is a ball $B_R(x_0)$, the inequality with \eqref{uR} follows from the scaling $y=(x-x_0)/R$ for $x\in B_R(x_0)$ and \eqref{nonzerobdn} for $B_1(0)$. 
\end{proof}

\textbf{Acknowledgement.} The authors would like to thank Tuoc Phan and Akif Ibragimov for their suggestions and helpful discussions.
L. H. acknowledges the support by NSF grant DMS-1412796.

\myclearpage

\def\cprime{$'$}

 \bibliographystyle{abbrv}

\end{document}